\documentclass[11pt]{amsart}
\usepackage[textheight=615 pt, textwidth=360pt, hmarginratio={1:1}]{geometry}

\usepackage{ xcolor} 
\usepackage{amsmath, amsthm}
\usepackage{amssymb}
\usepackage{array}
\usepackage{amsfonts}
\usepackage{graphicx}
\usepackage[normalem]{ulem}
\usepackage{tikz-cd}
\usepackage{hyperref}
\usepackage[capitalise]{cleveref}
\usepackage{mathdots}
\usepackage{framed}
\usepackage{enumitem}
\usepackage{caption}
\crefname{equation}{Equation}{Equations}
\crefname{conjecture}{Conjecture}{Conjectures}
\crefname{figure}{Figure}{Figures}
\crefname{thm}{Theorem}{Theorems}

\newtheorem{thm}{Theorem}
\numberwithin{thm}{section} 
\newtheorem{lemma}[thm]{Lemma}
\newtheorem{cor}[thm]{Corollary}
\newtheorem{prop}[thm]{Proposition}

\theoremstyle{definition}
\newtheorem{defn}[thm]{Definition}

\theoremstyle{remark}
\newtheorem{rem}[thm]{Remark}

\newtheorem{example}[thm]{Example}

\def\n{\ensuremath{e}}

\newcommand{\IQ}{\mathbb{Q}}

\newcommand{\IZ}{\mathbb{Z}}
\newcommand{\Pet}{\mathbf{P}}
\newcommand{\excise}[1]{}
\newcommand\junk[1]{}

\newcommand{\Fl}{\mathcal Fl}
\newcommand{\Hom}{\operatorname{Hom}}

\newcommand{\set}[2]{\left\{#1\,\middle\vert\,#2\right\}}
\def\bar{\overline}
\newcommand{\X}{\mathfrak X}

\begin{document}
\title{Positivity of Peterson Schubert Calculus}
\author{Rebecca Goldin}
\address{
Department of Mathematics,
George Mason University,
Fairfax VA 22030
USA}
\email{rgoldin@gmu.edu}

\author{Leonardo Mihalcea}
\address{
Department of Mathematics, 
Virginia Tech University, 
Blacksburg, VA 24061
USA
}
\email{lmihalce@vt.edu}

\author{Rahul Singh}
\address{
Department of Mathematics, 
Virginia Tech University, 
Blacksburg, VA 24061
USA
}
\email{rahul.sharpeye@gmail.com}

\maketitle
\begin{abstract} 
The Peterson variety is a subvariety of the flag manifold $G/B$ 
equipped with an action of a one-dimensional torus,
and a torus invariant paving by affine cells, called Peterson cells.
We prove that the equivariant pull-backs of Schubert classes indexed by
arbitrary Coxeter elements are dual (up to an intersection multiplicity)
to the fundamental classes of Peterson cell closures.
Dividing these classes by the intersection multiplicities yields a $\mathbb Z$-basis for the equivariant cohomology of the Peterson variety.
We prove several properties of this basis, including a Graham positivity property
for its structure constants, and stability with respect to inclusion in a larger Peterson variety.
We also find formulae for intersection multiplicities with Peterson classes.
This explains geometrically, in arbitrary Lie type,
recent positivity statements proved in type A by Goldin and Gorbutt.
\end{abstract}

\section{Introduction}
Among the most important properties of the cohomology rings of complex flag manifolds 
is that they have a distinguished basis: the Schubert basis, $\{\sigma_v: v\in W\}$, indexed by the Weyl group $W$. The Schubert structure constants for the product in singular cohomology in this basis are {\em positive}, which means that all nonzero coefficients in the product expansion
$$
\sigma_u\sigma_v =\sum_{w\in W} c_{uv}^w \sigma_w
$$ are positive. This is a consequence of transversality
properties, as the coefficients $c_{uv}^w$
count the number of points in the 
intersection of three Schubert varieties translated in general (transversal)
position. The study of positivity properties of these coefficients across a wide range of cohomology theories has spurred a large body of literature in algebraic geometry, representation theory, and combinatorics. 

In the case of torus-equivariant cohomology, a conjecture by Peterson \cite{peterson:notes}, proved by
Graham \cite{graham:positivity}, states that the
Schubert structure constants for
the torus equivariant cohomology rings of flag manifolds are polynomials
in the simple roots with positive coefficients. Graham's proof relies on refined
transversality techniques (see also \cite{AGM:KTpos}), which are expected to be useful
in analyzing the equivariant cohomology ring of varieties related to flag manifolds.

Hessenberg varieties form a remarkable family of subvarieties of flag manifolds, 
and appear across multiple disciplines within mathematics;
see~\cite{Abe:survey} for a recent survey. 
In this paper we focus on a particular class of 
regular nilpotent Hessenberg varieties, namely Peterson varieties.

Peterson varieties may also be realized as a flat degeneration of 
certain regular semisimple Hessenberg varieties such as the permutohedral 
variety; see e.g., \cite{klyachko85,klyachko95,abe.fujita.zeng}.
Peterson varieties share many properties with flag varieties, 
and provide a fertile ground for exploration. 
They initially appeared in the study of the 
quantum cohomology ring of (generalized)
flag manifolds $G/B$ in \cite{kostant:flag,peterson:notes}.

The purpose of this paper is to prove a positivity property of the equivariant cohomology ring 
of Peterson varieties, similar to that of flag varieties. This may be seen as a step towards investigating positivity and transversality properties of the (equivariant) cohomology
ring of the larger family of Hessenberg varieties, little of which is known. We give next a more precise account of our results.

Let $G$ be a complex, semisimple Lie group,
let $B, B^- \subset G$ be opposite Borel subgroups,
and let $T:= B \cap B^-$ be the associated maximal torus.
Let $\Delta$ be the set of positive simple roots corresponding to the choice of $B$,
and let $e$ be a principal nilpotent element contained in $\bigoplus_{\alpha\in \Delta} \mathfrak g_\alpha$,
where $\mathfrak g_\alpha$ is the root space in $\mathfrak g:=Lie(G)$ corresponding to the root $\alpha$.
Denote by $W$ the Weyl group associated to $(G,T)$ with length function $\ell:W \to \mathbb N$, 
and denote by $w_0 \in W$ the longest element.
Let $G^e$ be the centralizer of $e$ in $G$.
The Peterson variety 
\begin{equation}\label{eq:defP}
\Pet:= \overline{G^e. w_0B} \hookrightarrow G/B, 
\end{equation}
is the closure of the $G^e$-orbit of $w_0B$ inside the flag manifold $G/B$. 
It admits an action of a one-dimensional torus $S \subset T$ with finitely many fixed points.
We denote by $\iota: \Pet \hookrightarrow G/B$ the inclusion.


In this manuscript we investigate $H^*_S(\Pet):=H^*_S(\Pet;\mathbb Z)$,
the integral $S$-equivariant cohomology ring of the Peterson variety.
A presentation of $H^*_S(\Pet;\IQ)$ by generators and relations was given by
Harada, Horiguchi and Masuda in \cite{harada.horiguchi.masuda:Peterson}.

Let $\sigma_{v_I} \in H^*_S(G/B)$ be a Schubert class indexed
by some Coxeter element $v_I \in W$ for $I \subset \Delta$, 
let $p_I:= \iota^*(\sigma_{v_I})$ be the pullback of $\sigma_{v_I}$ along the inclusion $\iota:\Pet\to G/B$,
and let $m(v_I)$ be the intersection multiplicity from \Cref{thm:dualsintro} below.
We show that $\left\{ \frac{p_I}{m(v_I)} \right\}_{I \subset \Delta}$ is a $H^*_S(pt)$-basis of $H_S^*(\Pet)$,
and that this basis is dual to the equivariant Borel-Moore homology basis
constituted by the fundamental classes of Peterson varieties (see \cref{se:cohomology}).
In particular, while the class $p_I$ depends on the choice of $v_I$,
the class $\frac{p_I}{m(v_I)}$ is independent of this choice.

In our main result we prove that the structure constants of multiplication
are positive with respect to the basis $\left\{ \frac{p_I}{m(v_I)} \right\}$
in the sense of Graham \cite{graham:positivity}. This generalizes recent
results in Lie type A by Goldin and Gorbutt \cite{goldin.gorbutt:PetSchubCalc}, 
who found manifestly positive combinatorial formulae
for the structure constants in question. Special cases of such formulae 
were found earlier by Harada and Tymoczko \cite{harada.tymoczko:Monk}, and by 
Drellich \cite{drellich:monk}. 

We now present a more precise version of our results. For $I\subset\Delta$, let 
$w_I$ be the maximal element of the Weyl group of $I$, and
let $\Pet_I^\circ := \Pet \cap Bw_IB/B$.
Tymoczko \cite{tymoczko:paving} and B{\u a}libanu \cite{balibanu:peterson} proved that
$\Pet_I^\circ \simeq \mathbb{C}^{|I|}$,
and the cells $\Pet_I^\circ$ (called \emph{Peterson cells}) form an affine paving of the Peterson variety.
Consequently, the fundamental classes $[\Pet_I]_S \in H_{2 |I|}^S(\Pet)$,
where $\Pet_I=\overline{\Pet_I^\circ}$, form a basis of $H_*^S(\Pet)$ over $H^*_S(pt)$.

For $v\in W$, let
 $X^v:=\overline{B^-vB/B}$ denote the (opposite) Schubert variety in $G/B$,
and let $\sigma_v \in H^{2\ell(v)}_S(G/B)$ be the corresponding Poincar\'e dual class, satisfying
the equality $\sigma_v\cap[G/B]_S=[X^v]_S$.
Consider the pairing
$\langle \cdot , \cdot \rangle: H^*_S(\Pet) \otimes H_*^S(\Pet) \to H^*_S(pt)$
of equivariant cohomology and equivariant homology defined by 
$\langle a,b \rangle = \int_{\Pet} a \cap b$; see \S \ref{sec:eqcoh}.
Our first result is the following (cf.~\cref{thm:duals} below):

\begin{thm}[Duality Theorem]
\label{thm:dualsintro}
Let $I, J$ be subsets of the set of simple roots $\Delta$ and let $v_I\in W$ be any Coxeter element for $I$.
Then
\[ \left\langle \iota^*\sigma_{v_I}, [\Pet_J]_S\right\rangle  = m(v_I) \delta_{I,J} \/, \]
where $m(v_I) \in \mathbb{Z}_{>0}$ is the multiplicity of the (unique) intersection point of $X^{v_I} \cap \Pet_I$.
\end{thm}
This follows because the varieties $X^{v_I}$ and $\Pet_I$ intersect at a unique point,
namely $w_I$, the longest element in the subgroup $W_I$ determined by $I$.
The remaining part of the proof exploits the poset structure of the affine paving by Peterson cells,
along with the duality of Schubert classes in $G/B$. 

A non-equivariant version of this theorem has appeared in the literature in type A in a recent preprint \cite{abe.horiguchi.kuwata.zeng:left-right}, and may be deduced in general Lie type from \cite{insko.tymoczko:intersection.theory, insko:schubert}, where the intersection $X^{v_I} \cap \Pet(I)$ was analyzed.
Besides working non-equivariantly, 
a key restriction in all these papers is that the Coxeter elements $v_I$ are not arbitrary,
but depend on a certain ordering of the simple roots.
Our 
approach removes this restriction.
In \cref{sec:intmult}, we give algorithms
to calculate the aforementioned multiplicities based on equivariant localization, 
and closed formulae for a particular basis; see also \Cref{intro:ClassicalFormulae} below 
and the related discussion. A formula for $m(v_I)$ for any Coxeter element $v_I$ 
was recently obtained in \cite{goldin.singh}; see \cref{generalFormula} below.

The duality theorem has several consequences.  
For each Coxeter element $v_I$ for $I$, recall that
$p_I:= \iota^*\sigma_{v_I} \in H^{2|I|}_S(\Pet)$.
Then the classes 
$\set{\frac{p_I}{m(v_I)}}{I \subseteq \Delta}$ form a 
$H^*_S(pt)$-basis of $H^*_S(\Pet)$;
see \cref{prop:Petbasis}.
By the duality theorem, 
the equivariant push forward $\iota_*:H_*^S(\Pet) \to H_*^S(G/B)$ is injective. 
Non-equivariantly, the injectivity was proved in \cite{insko.tymoczko:intersection.theory}.

The cocharacter $h$ of $T$ satisfying $\alpha(h) =2$ for all $\alpha \in \Delta$
determines a one dimensional subtorus $S \subset T$,
satisfying $\alpha|S=\alpha'|S$ for any $\alpha,\alpha'\in\Delta$; see \cref{subsec:petersonVariety}.
Consequently there is
a well defined element $t \in H^*_S(pt)$ given by
$t:=\alpha|S$ for $\alpha\in\Delta$.

\begin{thm}[Positivity]
\label{thm:mainintro}
Let $I, J, K$ be subsets of $\Delta$. 
The structure constants of multiplication, $c_{I,J}^K \in H^*_S(pt)$, given by
\begin{equation}
\label{E:Pexp}
p_I \cdot p_J = \sum_K c_{I,J}^K p_K 
\end{equation} 
are polynomials in $t$ with non-negative coefficients.
\end{thm}
\Cref{thm:mainintro} (\cref{thm:mainpos}~below) generalizes several positivity statements to arbitrary Lie type, 
while providing a uniform proof in all cases. 
In the case where $|I|=1$, i.e., $p_I$ is a divisor class,
a positive Monk-Chevalley formula for the structure constants $c_{I,J}^K$ was obtained
by Harada and Tymoczko \cite{harada.tymoczko:Monk} in Lie type A,
and in arbitrary Lie type by Drellich \cite{drellich:monk}.
For general $c_{I,J}^K$, and in Lie type A, Goldin and Gorbutt \cite{goldin.gorbutt:PetSchubCalc}
found a manifestly positive combinatorial formula
for all equivariant coefficients $c_{I,J}^K$ in the expansion \eqref{E:Pexp}.
A different combinatorial model computing these coefficients in
non-equivariant cohomology was recently obtained
in \cite{abe.horiguchi.kuwata.zeng:left-right}.


The proof of \Cref{thm:mainintro} relies on the duality theorem,
and on positivity statements proved by Graham \cite{graham:positivity}.
The structure constants of the multiplication $\sigma_u\cup\sigma_v \in H^*_S(G/B)$
are positive in the sense of \Cref{thm:mainintro}.
Hence it suffices to show that the coefficients $b_w^J \in H^*_S(pt)$ of the restricted classes 
$ \iota^*\sigma_w = \sum_J b_w^J p_J $ ($w \in W$ arbitrary) satisfy the same positivity.
By the duality theorem, 
the positivity of the coefficients $b_w^J$ is equivalent to the positivity (in a suitable sense)
of the coefficients $c_I^v$ in the Schubert expansion 
\begin{equation}
\label{E:Petexp}
\qquad \iota_*[\Pet_I]_S = \sum_{v\in W} c_I^v [X_v]_S \quad \in H_*^S(G/B) \/, 
\end{equation} 
where $[X_v]_S$ denotes the homology class of the Schubert variety $X_v:=\overline{BvB/B}$ in $G/B$.
In the non-equivariant case, this is clear. Indeed, by Kleiman transversality \cite{kleiman:transversality}, the Schubert classes are a basis for the Chow group of $G/B$, and form
a set of primitive generators for the cone of effective algebraic cycles. 
Since the Chow group 
is equal to $H_*(G/B)$ (cf.~\cite[Ex. 19.1.11]{fulton:IT}), the claim follows.
Equivariantly, we deduce the positivity of $c_I^v$ from a general positivity result of Graham
\cite{graham:positivity} for expansions of fundamental classes of torus invariant varieties;
cf.~\Cref{thm:hompos}. 
A different geometric approach to positivity (in the non-equivariant setting)
was pursued in \cite{abe.horiguchi.kuwata.zeng:left-right}; see \cref{rmk:posrmk} below.

Using equivariant localization,
we obtain formulae for the multiplicities $m(v_I)$ in the duality theorem,
and an effective algorithm to find the Schubert expansion from \cref{E:Petexp};
see \cref{schubertExpansion}. 
Our algorithms for the coefficients $c_I^v$, and for the multiplicities $m(v_I)$,
rely on the restriction of the Schubert classes $\sigma_{v_I}$ to the fixed points $w_J$ in $G/B$ (see \cref{se:cohomology}).
These {\em equivariant localizations} may be calculated using
formulae developed by Andersen, Jantzen, and Soergel \cite{AJS} and Billey \cite{billey:kostant}).
 
\begin{thm}
\label{intro:ClassicalFormulae}
\begin{enumerate}[label=(\alph*),leftmargin=*]
\item
Let $I$ be a connected 
Dynkin diagram with the standard labelling,
see \cite{bourbaki:Lie46}, and set $v_I=s_1s_2\cdots s_n$.
Then,
\begin{align*}
m(v_I)=
\begin{cases}
1      &\text{if }I=A_n,\\
2^{n-1}&\text{if }I=B_n, C_n,\\
2^{n-2}&\text{if }I=D_n,\\
72=2^3 \cdot 3^2     &\text{if }I=E_6,
\end{cases}
&&
m(v_I)=
\begin{cases}
864=2^5 \cdot 3^3     &\text{if }I=E_7,\\
51840=2^7 \cdot 3^4 \cdot 5  &\text{if }I=E_8,\\
48=2^4 \cdot 3     &\text{if }I=F_4,\\
6 = 2 \cdot 3     &\text{if }I=G_2.
\end{cases}
\end{align*}
\item
Let $I_1,\cdots,I_k$ be the connected components of a Dynkin diagram $I$,
and let $v_1,\cdots,v_k$ be any Coxeter elements for $I_1,\cdots,I_k$ respectively.
Then $v:=v_1\cdots v_k$ is a Coxeter element for $I$, and $m(v)=\prod m(v_j)$.
\end{enumerate}
\end{thm}
See \cref{ClassicalFormulae} below. The factorization is related
to the exponents of the Lie algebra of $G$, see \S \ref{sec:intmult} below, and also the recent paper
\cite{goldin.singh}.
\Cref{intro:ClassicalFormulae} generalizes a result of Insko \cite{insko:schubert},
who showed that when $I=A_n$, $m(s_1\cdots s_n)=1$.
More generally, the theorem addresses \cite[Question 1]{insko.tymoczko:intersection.theory} by providing an explicit formula for these multiplicities; in particular, it disproves the conjecture by  Insko and Tymoczko
that the multiplicities are always $1$ or $2$ in classical Lie types.
The proof of part (a) of \Cref{intro:ClassicalFormulae} utilizes equivariant localization,
while the proof of part (b) utilizes the stability property of Peterson classes explained below.
Using parts (a) and (b) concurrently allows us to compute $m(v_I)$ for some Coxeter element $v_I$ in each Dynkin diagram $I$,
and hence allows us to construct the dual class $\frac{\iota^*\sigma_{v_I}}{m(v_I)}$ of any Peterson subvariety $\Pet_I\subset\Pet$.

Consider $I\subset\Delta$, a subset of the Dynkin diagram, and
let $G_I$ be a semisimple group with Dynkin diagram $I$.
Let $G_I/B_I$ and $\Pet(I)$ denote the flag variety and the 
Peterson variety of $G_I$ respectively, and let $S_I$
be the  one-dimensional subtorus defined analogously to $S$, and acting on $\Pet(I)$.
There is a natural 
closed embedding $i:G_I/B_I \hookrightarrow G/B$, but unfortunately
there may not be a morphism $S_I \to S$ which is compatible with this 
embedding. This leads to some technical subtleties explained in Section 
 \ref{sec:petvarstab}. The upshot is that there is an algebra isomorphism
$H^*_S(pt;\mathbb{Q}) \simeq H^*_{S_I}(pt;\mathbb{Q})$, and induced maps 
$H^*_S(G/B;\mathbb{Q}) \xrightarrow{i^*} H^*_{S_I}(G_I/B_I;\mathbb{Q})$ and
$H_*^{S_I}(G_I/B_I;\mathbb{Q}) \xrightarrow{i_*} H_*^{S}(G/B;\mathbb{Q})$.
The stability theorem, proved in \Cref{cor:PSisPet,prop:naturality}, is the following.
\begin{thm}[Stability]
\label[theorem]{intro:stable} 
(a) $i(\Pet(I))=\Pet \cap i( G_I/B_I) =\Pet_I$, as  subsets of $G/B$.


(b) For $J\subset I$, we have $i_*([\Pet(J)]_{S_I})=[\Pet_J]_S$,
as classes in $H_*^S(\Pet; \mathbb{Q})$.

(c) Let 
$j: \Pet(I)\hookrightarrow \Pet$ be the restriction of $i:G_I/B_I \to G/B$.
For $K\subset\Delta$, we have 
$$j^*(p_K)=\begin{cases}p_K&\text{if }K\subset I,\\0&\text{otherwise},\end{cases}$$
as classes in $H^*_{S_I}(\Pet(I);\mathbb{Q})$.

Furthermore, in the non-equivariant case, the statements in (b) and (c) hold over $\mathbb{Z}$.
\end{thm}

The proof of the stability theorem utilizes a common alternate description of the Peterson variety, namely
\begin{equation}
\label{eq:seconddefinition}
\Pet=\set{ gB\in G/B}{ Ad(g^{-1})e\in \mathfrak b \oplus \bigoplus_{\alpha\in \Delta} \mathfrak g_{-\alpha}},
\end{equation}
where $\mathfrak b=Lie(B)$.
In Appendix A,
we take the opportunity to present a proof that the definitions 
\eqref{eq:defP} and \eqref{eq:seconddefinition} are equivalent,
a matter of folklore implied by, and implicit in, Kostant's original work \cite{kostant:flag}.

\emph{Acknowledgements:}  The authors thank an anonymous referee for
careful reading and suggestions which helped improve the exposition of this paper.
We thank Ana B{\u a}libanu 
for explaining to us a proof of the irreducibility of Peterson variety,
and  Hiraku Abe and Tatsuya Horiguchi for explaining to us their results from
\cite{abe.fujita.zeng,abe.horiguchi.kuwata.zeng:left-right,horiguchi2021mixed}.
LM would like to thank Sean Griffin for useful discussions and 
for pointing out the reference \cite{nadeau.tewari}.
RG thanks Julianna Tymoczko for insightful conversations about Peterson varieties. 
The computer calculations in this paper were coded in SageMath \cite{sagemath}.
We thank the authors of the Weyl Groups and Root Systems libraries in SageMath.
RS would like to thank Camron Withrow for his help with the Sage code. During the 
preparation stages of this paper, LM enjoyed the hospitality and the wonderful environment of ICERM, 
as a participant to the special semester in `Combinatorial Algebraic Geometry'. 

R.~F.~Goldin was supported in part by a National Science Foundation grant DMS-2152312. L.~C.~Mihalcea was supported in part by a Simons Collaboration Grant and 
by the National Science Foundation under grant DMS-2152294, and under grant
DMS-1439786 while the author was in residence at the Institute for Computational and Experimental Research in Mathematics in Providence, RI, during the Spring 2021 semester.

{\em Conventions.} We work over the field of complex numbers. By a variety
we mean a reduced, irreducible scheme of finite type. Schemes defined as algebraic group orbits,
and closures of group orbits, 
are always equipped with the induced reduced scheme structure; see, e.g., \cite[tag 01IZ]{stacks-project}.

\section{Equivariant (Co)Homology}\label{se:cohomology}\label{sec:eqcoh}

Let $X$ be a complex algebraic variety equipped with a left action of a torus $T$.
We recall aspects of the $T$-equivariant homology and cohomology of $X$.
We will use the Borel model of equivariant cohomology,
and equivariant Borel-Moore homology,
following the setup in Graham's paper \cite{graham:positivity}. 
We refer to \cite[Ch~19]{fulton:IT}, \cite[Appendix~B]{fulton:young}, \cite[\S 2.6]{chriss.ginzburg}
for more details about cohomology and Borel-Moore homology. 

Since we are working with algebraic varieties, 
our statements and proofs
could have been written using the language of equivariant Chow groups
\cite{edidin.graham}. For full results, this requires some additional properties of
the operational Chow ring of {\em linear varieties} proved by Totaro \cite{totaro}. 
Aware of this technicality, the reader may 
use the equivariant cycle map from \cite{edidin.graham} to 
freely swap between the Borel-Moore and Chow theories.

Fix an identification $T\cong(\mathbb{C}^*)^r$ and let 
$ET =(\mathbb{C}^\infty \setminus 0)^r$ be the universal $T$-bundle
with classifying space $BT= (\mathbb{P}^\infty)^r$.
The product $ET \times X$ has a right $T$-action given by $(e,y).t:= (et, t^{-1}y)$.
The action is free,
and the orbit space $X_T:= (ET \times X)/T$ is called the Borel mixing space of $X$.
The universal $T$-bundle $ET \to BT$ admits 
finite dimensional approximations $ET_n \to BT_n$,
where $ET_n =(\mathbb{C}^{n+1} \setminus 0)^r$ and $BT_n:=(\mathbb{P}^n)^r$.
These induce finite dimensional approximations of the Borel mixing space 
$X_{T,n} := (ET_n \times X)/T$,
and inclusions $X_{T,n_1} \subset X_{T,n_2}$ for $n_1 < n_2$. 

We define the equivariant cohomology ring by $H^*_T(X):= H^*(X_T)$;
note that we have $H^i_T(X)= H^i(X_{T,n})$ for sufficiently large $n$.
The equivariant Borel-Moore homology groups are defined via a limiting property,
\begin{align*}
&&&&H_i^T(X) := H_{i+ 2nr}^{BM}(X_{T,n}),&& \textrm{ for }  n \gg 0 &&
\end{align*}
where the right hand side is the ordinary Borel-Moore homology. 
If $V \subset X$ is a closed $T$-stable subvariety of $X$ of complex dimension $d$, 
its fundamental class $[V]_T$ is an element in $H_{2d}^T(X)$. The cap product
gives the  (total) equivariant homology 
$H_*^T(X) = \bigoplus_i H_{i}^T(X)$ a graded module structure over the equivariant
cohomology ring $H^*_T(X)$.

If $X=pt$, then $H^*_T(pt) = H^*(BT)$ is naturally identified with the 
symmetric algebra $\operatorname{Sym}\X(T)$ of the character group 
$\X(T):=\Hom(T,\mathbb C^*)$ of $T$ (written additively).
For any map $S\to T$ of tori, 
we have a natural map of algebras $H^*_T(X) \to H^*_S(X)$,
compatible with the algebra map $H^*_T(pt)\to H^*_S(pt)$ induced by $\X(T) \to\X(S)$.
Taking $S$ to be the trivial subgroup in $T$, we obtain a ring homomorphism 
$
H^*_T(X)\to H^*(X)
$.
(One can show that this map is surjective for spaces with affine pavings in the sense of
\Cref{lemma:generate} below;
we will not need this fact.)

The morphism $X_T \to BT$ that projects onto the first factor
gives the equivariant cohomology $H^*_T(X)$ the structure
of a graded algebra over $H^*_T(pt)$.
In addition, the cap product $\cap$ endows the 
equivariant homology $H_*^T(X)$ with a graded module structure over $H^*_T(X)$. 
Equivalently, there is a compatibility of cap and cup products given by
$(a \cup b) \cap c = a \cap (b \cap c)$, for $a,b \in H^*_T(X)$, $c \in H_*^T(X)$.

Each irreducible, $T$-stable, closed subvariety $Z \subset X$ of complex dimension 
$k$ has a fundamental class $[Z]_T \in H_{2k}^T(X)$. If $X$ is smooth and irreducible,
then there exists a unique (Poincar{\'e} dual) class $\eta_Z \in H^{2 (\dim X - k)}_T(X)$
such that $\eta_Z \cap [X]_T = [Z]_T$. 

Any $T$-equivariant morphism of $T$-varieties 
$f:X \to Y$ induces a degree preserving pull-back morphism of $H^*_T(pt)$-algebras 
$f^*: H^i_T(Y) \to H^i_T(X)$. 
For a point $x\in X$ fixed by the $T$ action, the inclusion $\iota_x: \{ x \} \to X$ induces a \emph{localization}
map $\iota_x^*: H^*_T(X) \to H^*_T(\{x\}) = H^*_T(pt)$.

If $f$ is proper then there is a push-forward $f_*:H_i^T(X) \to H_i^T(Y)$,
defined as follows. Let $Z \subset X$ be closed, irreducible and $T$-stable. 
Then $f_*[Z]_T = d_Z [f(Z)]_T$ if $\dim f(Z)=\dim Z$,
where $d_Z$ is the generic degree of the restriction $f:Z \to f(Z)$, and
$f_*[Z]_T=0$ if $\dim f(Z) < \dim Z$. The push-forward and pull-back are related by the usual projection
formula $f_* (f^*(a) \cap c) = a \cap f_*(c)$.

An important particular case is when $X$ 
is complete, thus $f: X \to pt$ is proper. 
For a homology class $c \in H_*^T(X)$, we denote by 
$\int_X c$ the class $f_*(c) \in H_*^T(pt)$.~
\footnote{We note that $f_*(c)$ agrees with ``integration over the fiber" when $X$ is smooth, justifying the notation.}
Recall that the equivariant homology $H^T_*(pt)$ of a point is a free $H^*_T(pt)$-module with basis $[pt]_T$.
Therefore we identify $H^*_T(pt) = H_*^T(pt)$ via the map $a\mapsto a\cap[pt]_T$.
Then we may define a pairing,
 \begin{equation}\label{E:pairingdef} \langle \cdot , \cdot \rangle : H^*_T(X) \mathop\otimes\limits_{H^*_T(pt)} 
 H_*^T(X) \to H^*_T(pt) \/; \quad 
\langle a, c \rangle := \int_X a \cap c \/. \end{equation}
We often abuse notation and for a {\em cohomology} class $a \in H^*_T(X)$ we write 
$\int_X a$ to mean $\int_X (a \cap [X]_T)$.

Following \cite[Ex~1.9.1]{fulton:IT} (see also \cite{graham:positivity}) we say that 
a $T$-variety $X$ admits a {\em $T$-stable affine paving}
if it admits a filtration $X:=X_n \supset X_{n-1} \supset \ldots$ by closed $T$-stable subvarieties
such that each $X_i \setminus X_{i-1}$ is a finite disjoint union
of $T$-invariant varieties $U_{i,j}$ isomorphic to affine spaces $\mathbb A^i$. 
The following has been proved by Graham; see \cite[Prop~2.1]{graham:positivity}. 
\begin{lemma}
\label{lemma:generate} Assume $X$ admits a $T$-stable affine paving, with cells $U_{i,j}$. 
\begin{enumerate}[label=(\alph*)]
\item
The equivariant homology $H_*^T(X)$ is a free $H^*_T(pt)$-module with 
basis $\{[\overline{U_{i,j}}]_T\}$.
\item
If $X$ is complete, the pairing from \Cref{E:pairingdef} is perfect, 
and so we may identify
$H^*_T(X) = Hom_{H^*_T(pt)}(H_*^T(X), H^*_T(pt))$.
\end{enumerate}
\end{lemma}


\section{Flag manifolds and Peterson Varieties}\label{sec:preliminaries}
In this section we recall some basic definitions about flag manifolds,
Schubert varieties, and Peterson varieties.
We mostly follow the setup in 
\cite{tymoczko:paving} and \cite{balibanu:peterson},
from which we will need several important results.

\subsection{Flag manifolds and Schubert varieties}
\label{sec:Schubert}
Fix a complex semisimple Lie group $G$, a Borel subgroup $B \subset G$,   $B^-\subset G$ an opposite Borel subgroup,
and  let $T:= B \cap B^-$ be a maximal torus. 
Denote by $\Delta$ the system of simple positive roots associated to $(G,B,T)$
and by $\Phi_\Delta^+ \subset \Phi_\Delta$ the set of positive roots included in the set of all roots.
The Weyl group $W:= N_G(T)/T$ is generated by simple reflections $s_i:= s_{\alpha_i}$ 
where $\alpha_i \in \Delta$.
Let $\ell: W \to \mathbb{Z}_{\ge 0}$ be the length function and $w_0$ the longest element in $W$. 
Then $B^- = w_0 B w_0$.

Any subset $I\subset \Delta$ determines a Weyl subgroup $W_I := \langle s_i: \alpha_i \in I \rangle$
and a corresponding standard parabolic subgroup $P_I$.
We denote by $w_I$ the longest element of $W_I$.
The flag manifold $G/B$ is a  smooth algebraic variety  
of complex dimension $\ell(w_0)$ with a transitive
action of $G$ given by left multiplication.
The flag manifold has a stratification into finitely many $B$-orbits, respectively $B^-$-orbits,
called the {\em Schubert cells}: $X_w^\circ:= BwB/B \simeq \mathbb{C}^{\ell(w)}$
and $X^{w,\circ}:= B^- wB/B \simeq \mathbb{C}^{\ell(w_0w)}$;
we have
\begin{equation}
\label{schubStrat}
G/B = \bigsqcup_{w \in W} X_w^\circ =  \bigsqcup_{w \in W} X^{w,\circ} \/. 
\end{equation} 
The closures $X_w:=\overline{X_w^\circ}$ and $X^w:=\overline{X^{w,\circ}}$ are called {\em Schubert varieties}  and {\em opposite Schubert varieties}, respectively.
The \emph{Bruhat order} is a partial order on $W$ characterized by inclusions of Schubert varieties  and opposite Schubert varieties. In particular, 
 $X_v \subset X_w$ if and only if $v \le w$, 
and $X^w\subset X^v$ if and only if $v\leq w$.
Following \Cref{lemma:generate},  
the homology classes
$\set{[X_v]_T}{v\leq w}$ form a basis of $H_*^T(X_w)$,
while  $\set{[X^v]_T}{w\leq v}$ form a basis of $H_*^T(X^w)$.

The cohomology classes $\sigma_v\in H_T^*(X)$ Poincar\'e dual to the $[X^v]_T$,
i.e. characterized by the equation $\sigma_v\cap[G/B]_T=[X^v]_T$, are called {\em Schubert classes}.
Note that \Cref{lemma:generate} also implies  $\set{\sigma_v}{v\in W}$ is a basis of $H_T^*(G/B)$ as a module over $H_T^*(pt)$.
Under the pairing in \cref{E:pairingdef},
the basis $\set{\sigma_v}{v\in W}$ is dual to the basis $\set{[X_v]_T}{v\in W}$,
i.e, we have 
$
\left\langle \sigma_v,[X_w]_T\right\rangle=\delta_{v,w}.
$

\subsection{The Peterson variety and Peterson cells}
\label{subsec:petersonVariety}
The Peterson variety appeared in the unpublished work of Peterson \cite{peterson:notes},
in relation to the quantum cohomology of $G/B$;
we refer the reader to \cite{kostant:flag,rietsch:totally} for details.

We recall the definition of the Peterson variety. 
Let $\mathfrak{g}:= Lie(G)$, $\mathfrak h:=Lie(T)$,
and consider the Cartan decomposition
\begin{align*}
\mathfrak{g} =\mathfrak{h} \oplus \bigoplus\limits_{\alpha \in \Phi_\Delta} \mathfrak{g}_\alpha.
\end{align*}
 For each simple root $\alpha \in \Delta$,
choose a root vector $e_\alpha \in \mathfrak{g}_\alpha$,
and let \begin{align*}e:= \sum_{\alpha\in \Delta} e_\alpha.\end{align*}
The element $e$ is a regular nilpotent element in the Lie algebra $\mathfrak b$ of $B$;
see \cite{kostant:principal} or \cite[Thm~4.1.6]{mcgovern.collingwood:nilpotent.orbits}.
Denote by $G^e \subset G$ the stabilizer of $e$ for the adjoint action of $G$ on $\mathfrak g$.
We have $G^e=(G^e)^\circ\times Z(G)$,
where $(G^e)^\circ$ is the identity component of $G^e$, and $Z(G)$ the center of $G$.
The identity component $(G^e)^\circ$ is a subgroup of the unipotent radical $U$ of $B$,
isomorphic to the affine variety $\mathbb{C}^n$,
where $n := |\Delta|$ is the number of simple roots,
i.e., the rank of $G$, cf. \cite[Cor~5.3]{kostant:principal}.
For instance, if $G = \mathrm{SL}_n(\mathbb{C})$,
then $(G^e)^\circ$ is the subgroup of upper triangular unipotent matrices with equal entries along each superdiagonal.
The Peterson variety is defined by
\begin{align}
\label{defn:Pet}
\Pet:= \overline{G^e.w_0B} \subset G/B \/.
\end{align}
This is an irreducible subvariety of $G/B$ of dimension $|\Delta|$, singular in general.

For any $\omega\in\mathfrak h$ contained in the coroot lattice,
the map $\varphi_\omega:\mathbb C\to \mathfrak h$ defined by 
$\varphi_\omega(z)= z\omega$ lifts to a cocharacter 
$exp(\varphi_\omega):\mathbb C^*\to T$. 
(Here the differential of $exp({\varphi_\omega})$ is equal to $\varphi_\omega$.
{In complex differential geometry, the map $exp({\varphi_\omega})$ intertwines with the (non-algebraic) exponential maps $exp:\mathbb C\to\mathbb C^*$ and $exp:\mathfrak h\to T$;
the cocharacter $exp({\varphi_\omega})$ is itself an algebraic map.)
This identifies the coroot lattice of $\mathfrak h$ with a subset of the cocharacters of $T$. 
See, e.g., \cite[Ch.~3,~Prop.~1.15]{vinberg} (in the algebraic setting), or \cite[p. 373-4]{fulton.harris:RTbook} (in the manifold setting).

We take $h=\sum_{\alpha\in\Phi_\Delta^+}\alpha^\vee$ to be the sum of the positive coroots, and denote by $S\subset T$ the image of the cocharacter corresponding to $h$.
Following \cite[Ch~6, Prop 29]{bourbaki:Lie46}, we have $\alpha(h)=2$ for all $\alpha\in\Delta$,
because $h$ is equal to twice the sum of the fundamental coweights.
In particular, it follows that $\alpha|S=\alpha'|S$ for any $\alpha,\alpha'\in\Delta$.
We set $t:=\alpha|S\in\X(S)\subset H^*_S(pt)$.

\begin{example}
Consider $G=SL_n$, and let $T\subset G$ be the set of diagonal matrices:
\begin{align*}
T=\set{\begin{pmatrix}z_1&0&0\\0&\ddots&0\\0&0&z_n\\\end{pmatrix}}{z_1\cdots z_n=1}.
\end{align*}
The $\alpha_i$, $1\leq i\leq n-1$, given by $\alpha_i(diag(z_1,\cdots,z_n))\mapsto z_i/z_{i+1}$,
form a set of simple roots.
The coroot $h$ corresponds to the one-dimensional subtorus $S$ given by
\begin{align*}
S=\set{\begin{pmatrix}z^{n-1}&0&0&0\\0&z^{n-3}&0&0\\0&0&\ddots&0\\0&0&0&z^{-n+1}\end{pmatrix}}{z\in\mathbb C^*}.
\end{align*}
The character $t$ of $S$ is the map given by $diag(z^{n-1},z^{n-3},\cdots,z^{-n+1})\mapsto z^2$.
\end{example}
}

\begin{rem}
The element $t$ need not be a generator of the ring $H^*_S(pt)$.
For example, if $G=SL_2$, we have
\begin{align*}
S=\set{\begin{pmatrix}z&0\\0&z^{-1}\end{pmatrix}}{z\in\mathbb C^*},
&& t:\begin{pmatrix}z&0\\0&z^{-1}\end{pmatrix}\mapsto z^2.
\end{align*}
The character group $\X(S)$ is generated by $t/2$,
which is the map 
\begin{align*}
\begin{pmatrix}z&0\\0&z^{-1}\end{pmatrix}\mapsto z.
\end{align*}
However, we always have 
either $H^*_S(pt)=\mathbb Z[t]$, or $H^*_S(pt)=\mathbb Z[t/2]$.
\end{rem}

 Since $[h,e_\alpha]=2e_\alpha$ for each simple root $\alpha$, we have $[h, e] = 2e$,  from which we observe that
$S$ normalizes $G^e$,
cf. \cite[Theorem 10]{kostant:flag},
resulting in 
an action of  the semidirect product $S \ltimes G^e$
on the Peterson variety. 

The following was proved in classical types by Tymoczko \cite[Thm~4.3]{tymoczko:paving}
and generalized to all Lie types by Precup \cite{precup}.

\begin{prop}\label{prop:vectorSpace}
For $I \subset \Delta$, let $w_I$ denote the longest element in the Weyl subgroup $W_I$.
\begin{enumerate}[label=(\alph*),leftmargin=*]
\item
The intersection $\Pet \cap BwB/B$ is nonempty if and only if $w=w_I$ for some subset $I\subset \Delta$.
\item 
The  set theoretic intersection $\Pet_I^\circ:=\Pet \cap Bw_IB/B$ is an affine space of dimension $|I|$.
In particular, its closure $\Pet_I$ is an irreducible subvariety of $X_{w_I}$.
\end{enumerate}
\end{prop}
Some of the details proving part (b) are implicit in \cite{balibanu:peterson}.
We take the opportunity to make these details explicit in \cref{prop:paving} below.
We will refer to $\Pet_I^\circ$ as a \emph{Peterson cell}; its closure 
$\Pet_I \subset X_{w_I}$ is an irreducible variety, 
and the Schubert cell decomposition of Schubert varieties yields an affine paving
\[\Pet_I = \bigsqcup_{J \subset I}\Pet_J^\circ \/. \] 
Following \cite[Prop 2.1(a)]{graham:positivity},
the classes $\set{[\Pet_I]_S}{I\subset\Delta}$ form a basis of $H_*^S(\Pet)$.
Observe that $S \subset T$ is a regular subtorus,
thus the fixed point loci for $S$ and $T$ in $G/B$ coincide, i.e.,
$(G/B)^T = (G/B)^S$; see e.g. \cite[\S 24.2, \S 24.3]{humphreys:linear}.
It follows that
\[\Pet^S = (G/B)^S \cap\Pet = (G/B)^T\cap\Pet= \{ w_I: I \subset \Delta \} \/, \]
where we utilize the usual identification $(G/B)^T = W$.

For $I \subset \Delta$, an element $v\in W$ is called a {\em Coxeter element} for $I$
if $v=s_{\alpha_1}\cdots s_{\alpha_k}$ for some enumeration $\alpha_1,\ldots,\alpha_k$ of $I$.
Recall the following result, cf.~\cite[Lemma 7]{insko.tymoczko:intersection.theory}:

\begin{prop}
\label{onePointIntersection}
Let $v_I$ be a Coxeter element for some subset $I\subset \Delta$.
Then the intersection $X^{v_I}\cap\Pet_I$ is the single (possibly non-reduced) point $w_I$.
\end{prop}
\begin{proof}
The intersection $Y:=X^{v_I}\cap\Pet_I$ is proper and $S$-stable.
Any fixed point in $Y^S\subset \Pet_I^S$ 
is of the form $w_J$, for some $J\subset I$.
On the other hand, 
since $w_J \in X^{v_I}$, we have $w_J \ge v_I$. 
Since $v_I$ is a Coxeter element for $W_I$,
$I \subset J$, and so $I=J$.
Thus $Y$ contains a unique $S$-fixed point;
hence by \cite[Prop~13.5]{borel:linear}, we have $Y=\{w_I\}$.
\end{proof}

\begin{cor}
\label{intDual}
Let $\eta_I\in H^*_S(G/B)$ be the Poincar\'e dual of $[\Pet_I]_S \in H_*^S(G/B)$,
$v_I$ a Coxeter element for $I$,
and $\tau_{w_I}$ the Poincar\'e dual of the point class $[w_I]_S$.
Then
\begin{align*}
&&\sigma_{v_I}\cup\eta_I = m(v_I) \tau_{w_I}
&&\text{and}&&
\int_{G/B}\sigma_{v_I}\cup\eta_I = m(v_I),&&
\end{align*}
where $m(v_I)$ is the multiplicity of $w_I$ in the intersection $X^{v_I} \cap \Pet_I$.
\end{cor}
\begin{proof}
Observe from \cref{lemma:generate} that $H^*_S(G/B)$ is torsion-free,
and hence the localization map 
$H^*_S(G/B)\to\mathop\oplus\limits_{w\in W} H^*_S(w)$
is injective (over $\mathbb Z$);
see \cite[Cor 1.3.2, Thm 1.6.2]{GKM} and \cite[Thm 3.1]{hsiang}.
By \cref{onePointIntersection}, the only potentially non-zero localization of $\sigma_{v_I}\cup\eta_I$ is at $w_I$, and
therefore $\sigma_{v_I}\cup\eta_I=m(v_I) \tau_{w_I}$ for some integer $m(v_I)$.
Under the specialization $H^*_S(G/B)\to H^*(G/B)$,
the class $\tau_{w_I}$ maps to $1\in H^*(G/B)$.
It now follows from \cite[Eq~(31)]{fulton:young}
that $m(v_I)$ is the multiplicity of the intersection $X^{v_I} \cap \Pet_I$.
\end{proof} 
In \cref{sec:intmult}, we provide
a formula for $m(v_I)$ based on equivariant localization,
and compute the value of $m(v_I)$ for certain Coxeter elements $v_I$.

\section{Poincar{\'e} duality and consequences}
\label{se:duality}
Let $G$ be a complex semisimple group,
and $\iota:\Pet \hookrightarrow G/B$ the corresponding Peterson variety,
as in \cref{sec:preliminaries}.
In \cref{thm:duals}, we construct a basis $\{p_I\}_{I\subset\Delta}$ of $H^*_S(\Pet)$
dual (up to scaling) to the basis $\{[\Pet_I]_S\}_{I\subset\Delta}$ of $H_*^S(\Pet)$.
\Cref{thm:duals} relates the  Schubert expansion of a Peterson class $[\Pet_I]_S$
to the expansion in the $\{p_I\}$ basis of the pull-backs $\iota^*\sigma_w$;
the latter can computed using equivariant localization and Gaussian elimination.
We sketch an example in \cref{schubertExpansion}.

\subsection{Peterson classes and duality}
\begin{lemma}
\label{lemma:order}
Let $I \subset \Delta$, and consider the expansion
\begin{align*}
\iota_*[\Pet_I]_S = \sum_{v\in W} c_I^v\, [X_v]_S \quad \in H_*^S(G/B). 
\end{align*}
Then $c_I^v =0$ unless $v \le w_I$.
\end{lemma}
\begin{proof} By \Cref{lemma:generate}, the equivariant homology 
$H_*^S(X_{w_I})$ has a $H^*_S(pt)$-basis
given by the fundamental classes $[X_v]_S$, where $v \le w_I$.
Since $\Pet_I$ is a subvariety of $X_{w_I}$, 
we have $\iota_*[\Pet_I]_S = \sum\limits_{v\leq w_I} c_I^v\, [X_v]_S$,
for some $c_I^v\in H^*_S(pt)$.
\end{proof}

\begin{lemma}
\label{prop:nonvan}
Let $I \subset \Delta$, and consider the expansion
\begin{align*}
\iota_*[\Pet_I]_S = \sum_{v\in W} c_I^v\, [X_v]_S \quad \in H_*^S(G/B).
\end{align*}
If $v$ is a Coxeter element for $J\neq I$, then $c_I^v=0$.
\end{lemma}
\begin{proof}
Suppose $v$ is a Coxeter element for some subset $J\subset\Delta$ for which $c_I^v\neq 0$.
Following \Cref{lemma:order}, we have $v\le w_I$, hence $J\subset I$.
On the other hand, since the expansion is homogeneous,
we have $|J|=\ell(v) \ge \dim \Pet_I = |I|$, and hence $J=I$.
\end{proof}

\begin{thm}[Duality Theorem]
\label[theorem]{thm:duals}
Let $I, J$ be subsets of the set of simple roots $\Delta$,
and let $v_I$ be a Coxeter element for $I$.
We have 
\begin{align*}
\langle \iota^*\sigma_{v_I}, [\Pet_J]_S \rangle = m(v_I) \delta_{I,J},
\end{align*}
where $m(v_I)$ is the multiplicity of the (unique) intersection point of $X^{v_I} \cap \Pet_I$.  In particular, $m(v_I)$ is a positive integer.
\end{thm}
\begin{proof}
Consider the Schubert expansion $\iota_*[\Pet_J]_S = \sum c_J^v [X_v]_S$.
Then 
\begin{align*}
\langle \iota^*\sigma_{v_I}, [\Pet_J]_S \rangle= \langle \sigma_{v_I}, \iota_* [\Pet_J]_S \rangle= c_J^{v_I},
\end{align*}
since the set $\{\sigma_v\}_{v\in W}$ forms a dual basis to the fundamental classes $\{[X_v]_S\}_{v\in W}$.
It follows from from \Cref{prop:nonvan} that $c_J^{v_I}=0$ for $I\neq J$.
For $J=I$, \cref{intDual} implies
\begin{align*}
c_I^{v_I}=\langle \sigma_{v_I}, \iota_* [\Pet_I]_S \rangle=\int_X\sigma_{v_I}\cup\eta_I=m(v_I)> 0.
\end{align*}
Finally, $m(v_I) \in \mathbb{Z}_+$ because the pairing~\eqref{E:pairingdef} has values in integral cohomology.
\end{proof}

We record the following consequence of the Duality theorem.

\begin{cor}
\label{prop:Petbasis}
For each $I\subset\Delta$, fix a Coxeter element $v_I$,
and set $p_I:= \iota^*\sigma_{v_I}  \in H^*_S(\Pet)$.
\begin{enumerate}[label=(\alph*)]
\item
The classes $\set{\dfrac{p_I}{m(v_I)}\in H^*_S(\Pet)}{I\subset\Delta}$ form a $H^*_S(pt)$-basis of $H^*_S(\Pet)$.
\item
The map $\iota_*:H_*^S(\Pet)\to H_*^S(G/B)$ is injective.
\end{enumerate}
\end{cor}
\begin{proof}
By \cref{thm:duals}, the classes $\frac{p_I}{m(v_I)}$ are dual to the classes $[\Pet_I]_S$,
and part (a) follows from \Cref{lemma:generate}.
For part (b), observe that the pairing
\begin{align*}
\left\langle\sigma_{v_J},\iota_*[\Pet_I]_S\right\rangle  = m(v_I) \delta_{I,J},
\end{align*}
along with the linear independence of the ${\sigma_{v_J}}$ in $H_S^*(G/B)$,
implies that the ${\iota_*[\Pet_I]_S}$ are linearly independent.
It follows that the map $\iota_*:H_*^S(\Pet)\to H_*^S(G/B)$ is injective.
\end{proof}
\begin{rem}
Part (a) of \Cref{prop:Petbasis} was proved in various generalities,
and for particular choices of Coxeter elements $v_I$,
in \cite{drellich:monk,insko.tymoczko:intersection.theory,abe.horiguchi.kuwata.zeng:left-right}.
The non-equivariant version of part (b) was proved in \cite[Thm~2]{insko.tymoczko:intersection.theory}.
\end{rem}

We also record the following immediate corollary,
which will be utilized in the proof of the positivity statement \cref{thm:hompos}.

\begin{cor}
\label[corollary]{thm:expansions}
For each $I\subset\Delta$, fix a Coxeter element $v_I$,
and set $p_I:= \iota^*\sigma_{v_I}  \in H^*_S(\Pet)$.
Consider the expansions
\begin{align*}
\iota^*\sigma_w   = \sum_{J\subset\Delta} b_w^J\, p_J,&& 
\iota_*[\Pet_I]_S = \sum\limits_{u\in W} c_I^u\, [X_u]_S \/.
\end{align*}
Then $c_I^u = m(v_I) b_u^I$~for all $u$, where $m(v_I)>0$ is the coefficient from the Duality \Cref{thm:duals}.
\end{cor}
\begin{proof}
Using Theorem \ref{thm:duals}
and the equality 
$\langle \sigma_v, [X_u]_S \rangle_{G/B} = \delta_{u,v}$,
we calculate,
\begin{align*}
c_I^u  &=   \langle \sigma_u, \iota_*[\Pet_I]_S\rangle_{G/B} 
	    = \langle \iota^*\sigma_u,  [\Pet_I]_S\rangle_{\Pet} 
        = m(v_I) b_u^I \/.
\end{align*}
Here the first equality follows from the definition of $c_I^u$, the second from the projection formula,
and the third from Theorem \ref{thm:duals} together with the definition of $b_u^I$. 
\end{proof}

\subsection{Schubert Expansion of the Peterson Classes}
\label{schubertExpansion}
In their study of certain regular Hessenberg varieties, 
Abe, Fujita and Zeng \cite{abe.fujita.zeng} found a beautiful closed formula for the
non-equivariant Schubert expansions of the fundamental classes of these varieties.
For the Peterson varieties discussed here, their formula states that
\[ \iota_*[\Pet] = \prod_{\alpha \in \Phi^+_\Delta \setminus \Delta} c_1(G \times^B \mathbb{C}_{-\alpha})  \cap [G/B]\/ \in H_*(G/B); \]
see Cor. 3.9 in {\em loc. cit.} However, since the line bundles $G \times^B \mathbb{C}_{-\alpha}$
are not globally generated, this formula involves cancellations.
A manifestly positive formula was recently found by Nadeau and Tewari \cite{nadeau.tewari},
and further investigated by Horiguchi \cite{horiguchi2021mixed}, in relation to {\em mixed Eulerian numbers}.
The origins of this approach lie in the realization of 
the Peterson variety as a flat degeneration of a smooth projective toric variety, 
called the (generalized) 
\emph{permutohedral variety}; see \cite{abe.fujita.zeng,nadeau.tewari}.
The permutohedral variety is a regular semisimple Hessenberg variety;
its cohomology ring has been classically studied e.g. by Klyachko \cite{klyachko85,klyachko95}.
In this section we present a different algorithm, which calculates the equivariant Schubert expansion 
of $\iota_*[\Pet]_S$.
The algorithm is based on \Cref{thm:expansions}, and it 
depends on the multiplicities $m(v_I)$ for some choice of Coxeter elements $v_I$, $I\subset \Delta$.
The values $m(v_I)$ for a particular such choice are computed in \cref{ClassicalFormulae}.
It would be interesting to 
utilize this algorithm to extend the formulae from \cite{abe.fujita.zeng,nadeau.tewari} to the equivariant 
setting; this will be left for elsewhere. 

\begin{prop}
\label{prop:SchubertExpansion}
Fix Coxeter elements $v_I$ for each subset $I\subset\Delta$,
and consider the matrices,
\begin{align*}
A_{u,I}=\iota^*_{w_I}\sigma_u,&&
C_{I,J}=\iota^*_{w_J}\sigma_{v_I},&&
M_{I,J}=m(v_I) \delta_{I,J}.
\end{align*}
Here $A$ is a 
$|W|\times 2^{|\Delta|}$ 
matrix,
and $C$ and $M$ are 
$2^{|\Delta|}\times 2^{|\Delta|}$ matrices.
The fundamental classes $[\Pet_I]_S$ and $[X_u]_S$ are related by the matrix equation,
\begin{equation}
\label{schubertExpansionFormula}
\left([\Pet_I]_S\right)_{I\subset\Delta}=(AC^{-1}M)^T\left([X_u]_S\right)_{u\in W}.
\end{equation}
\end{prop}

\begin{proof}
Consider the commutative diagram,
\begin{equation}
\label{localizationComDiag}
\begin{tikzcd}
H^*_S(G/B)\arrow[r,"\oplus\iota^*_u"]\arrow[d,"\iota^*"] & \bigoplus\limits_{u\in W}H^*_S(u)\arrow[d,twoheadrightarrow]\\
H^*_S(\Pet)\arrow[r,"\oplus\iota^*_{w_I}"]                   & \bigoplus\limits_{I\subset\Delta}H^*_S(w_I).
\end{tikzcd}
\end{equation}
Let $\mathcal Q$ be the fraction field of the integral domain $H^*_S(pt)$, and let $R_{\mathcal Q}:= R\otimes_{H^*_S(pt)} \mathcal Q$ for any $H^*_S(pt)$-module $R$.
The map
$H^*_S(\Pet)\to\bigoplus\limits_{I\subset \Delta} H^*_S(w_I)$
induces an isomorphism
$H^*_S(\Pet)_{\mathcal Q}\xrightarrow\sim\bigoplus\limits_{I\subset\Delta} H^*_S(w_I)_{\mathcal Q}$;
see \cite[Cor~1.3.2, Thm~1.6.2, Thm~6.3]{GKM}.
Observe that $H^*_S(\Pet)$ is torsion-free,
and is naturally identified as a lattice in the $\mathcal Q$-vector space $H^*_S(\Pet)_{\mathcal Q}$.
Let $\tau_I$ denote 
a generator of $H^*_S(w_I)$,
and consider the column vectors,
\begin{align*}
\boldsymbol\sigma & =\left(\iota^*\sigma_u\right)_{u\in W}, & \boldsymbol\tau & =\left(\tau_I\right)_{I\subset\Delta},\\
\mathbf p         & =\left(p_I\right)_{I\subset\Delta},     & \mathbf q       & =\left(\frac{p_I}{m(v_I)}\right)_{I\subset\Delta}.
\end{align*}
We have the following equalities in $H^*_S(\Pet)_{\mathcal Q}$:
\begin{align}
\label{schubertExpansionMatrices}
\mathbf p=M\mathbf q,&& \boldsymbol\sigma=A\boldsymbol\tau,&& \mathbf p=C\boldsymbol\tau.
\end{align}
The matrix $C$ is invertible since both $\{p_I\}_{I\subset\Delta}$ and $\{\tau_I\}_{I\subset\Delta}$ are bases for $H^*_S(\Pet)_{\mathcal Q}$.
We deduce that
$
\boldsymbol\sigma=AC^{-1}M\mathbf q
$.
\Cref{schubertExpansionFormula} now follows from \Cref{thm:expansions}.
\end{proof}

\begin{rem}
The coefficients $A_{w,I}$ and $C_{I,J}$ in \cref{prop:SchubertExpansion} can be computed by
composing the localization formula for the $T$-equivariant Schubert classes
(cf. \cite{AJS,billey:kostant}) with the restriction map $\X(T)\to\X(S)$
 defined by $\lambda \mapsto \lambda|{S}$.
\end{rem}

\begin{rem}
The invertibility of the matrix $C$ in \cref{prop:SchubertExpansion} can be directly deduced from the observation that
$\iota_{w_J}^*\sigma_{v_I}\neq 0$ if and only if $I\subset J$,
and hence $C$ is upper triangular with respect to the partial order $I\leq J\iff I\subset J$.
\end{rem}

\begin{example}
We use \cref{prop:SchubertExpansion} to compute the Schubert expansion of $[\Pet]_S$
in the case $\Delta=B_2$, with $v_\Delta=s_1s_2$.
Set
\begin{align*}
p_\phi=\iota^*\sigma_{id},&&
p_{\{1\}}=\iota^*\sigma_1,&&
p_{\{2\}}=\iota^*\sigma_2,&&
p_{\{1,2\}}=\iota^*\sigma_{12}.
\end{align*}
Composing the localization formula for Schubert classes (cf. \cite{AJS,billey:kostant}) with the restriction map $\X(T)\to\X(S)$,
we obtain the $S$-equivariant localizations of the Schubert classes: 
\begin{align}
\label{B2SchubertExpansion}
\begin{pmatrix}
\iota^*\sigma_{id}\\
\iota^*\sigma_{1}\\
\iota^*\sigma_{2}\\
\iota^*\sigma_{12}\\
\iota^*\sigma_{21}\\
\iota^*\sigma_{121}\\
\iota^*\sigma_{212}\\
\iota^*\sigma_{1212}
\end{pmatrix}
&=
\left(\begin{array}{rrrr}
1 & 1 & 1 & 1 \\
0 & t & 0 & 4t \\
0 & 0 & t & 3t \\
0 & 0 & 0 & 6t^2\\
0 & 0 & 0 & 6t^2\\
0 & 0 & 0 & 6t^3\\
0 & 0 & 0 & 6t^3\\
0 & 0 & 0 & 6t^4\\
\end{array}\right)
\begin{pmatrix}
\tau_\phi\\
\tau_1\\
\tau_2\\
\tau_{12}
\end{pmatrix}.
\end{align} 
The $8\times 4$ matrix in \cref{B2SchubertExpansion} corresponds to the matrix $A$ in \cref{schubertExpansionMatrices},
and the matrix $C$ is precisely its top $4\times 4$ submatrix.
The multiplicities $m(v_I)$ are  computed in \cref{ClassicalFormulae};
we have $m(v_I)=1$ for all $I\subsetneq B_2$ and $m(v_\Delta)=2$, i.e.,
\begin{equation*}
M=\begin{pmatrix}
1&0&0&0\\
0&1&0&0\\
0&0&1&0\\
0&0&0&2
\end{pmatrix}.
\end{equation*}
Applying \cref{schubertExpansionFormula}, we obtain 
\begin{equation*}
\begin{pmatrix}
\iota_*[\Pet_\phi]_S\\
\iota_*[\Pet_{\{1\}}]_S\\
\iota_*[\Pet_{\{2\}}]_S\\
\iota_*[\Pet]_S
\end{pmatrix}
=\begin{pmatrix}
1&0&0&0&0&0&0&0\\
0&1&0&0&0&0&0&0\\
0&0&1&0&0&0&0&0\\
0&0&0&2&2&2t&2t&2t^2\\
\end{pmatrix}
\begin{pmatrix}
\iota_*[X_{id}]_S\\
\iota_*[X_{1}]_S\\
\iota_*[X_{2}]_S\\
\iota_*[X_{12}]_S\\
\iota_*[X_{21}]_S\\
\iota_*[X_{121}]_S\\
\iota_*[X_{212}]_S\\
\iota_*[X_{1212}]_S
\end{pmatrix}.
\end{equation*}
In particular, we have $\iota_*[\Pet]_S=2[X_{12}]_S+2[X_{21}]_S+2t[X_{121}]_S+2t[X_{212}]_S+2t^2[X_{1212}]_S$.
\end{example}

\section{Positivity}
\label{sec:positivity}

We recall a theorem of Graham \cite[Thm.~3.2]{graham:positivity}, which plays a key role in the proof of our positivity results, \cref{thm:hompos,thm:mainpos}.

\begin{thm}
\label{thm:graham-pos}
Let $B'$ be a connected solvable group with unipotent radical $N'$,
and let $T' \subset B'$ be a maximal torus, so that $B' = T'N'$.
Let $\alpha_1, \ldots , \alpha_d$ be the weights of $T'$ acting on $Lie(N')$.
Let $X$ be a scheme with a $B'$-action, and $Y$ a $T'$-stable subvariety of $X$.
Then there exist $B'$-stable subvarieties $D_1, \ldots , D_k$ of $X$
such that in the equivariant homology $H_*^{T'}(X)$,
\begin{align*}
[Y]_{T'} = \sum f_i [D_i]_{T'},
\end{align*}
where each $f_i \in H^*_{T'}(pt)$ is a linear combination of monomials 
in $\alpha_1, \ldots , \alpha_d$ with non-negative integer coefficients.
\end{thm}

\begin{thm}
\label[theorem]{thm:hompos}
Let $I$ be a subset of $\Delta$,
let $\iota:\Pet \hookrightarrow G/B$ be the inclusion,
and consider the Schubert expansion,
\begin{align*}
\iota_*[\Pet_I]_S = \sum_{v\in W} c_I^v\, [X_v]_S.
\end{align*}
Then $c_I^v \in H^*_S(pt)$ is a polynomial in $t$ with non-negative coefficients.
\end{thm}
\begin{proof}
We apply Graham's positivity theorem to the following situation:
$Y=\Pet_I\subset X=G/B$, $T'=S$, and $B'=SU$,
where $U$ is the unipotent radical of $B$.
We have $U\subset B'\subset B$, and  since the $U$-orbits and $B$-orbits 
in $G/B$ coincide, the $B'$-orbits in $G/B$ are precisely the Schubert cells $X_v^\circ$.

Observe that 
the restriction map $\X(T)\to\X(S)$ is given by $\alpha\mapsto ht(\alpha)t$ for 
$\alpha\in\Phi^+_\Delta$, where $ht(\alpha)$ is the height of $\alpha$. It follows that
the weights for the $S$-action on $Lie(U)$ are positive integer multiples of $t$.
It follows from \cref{thm:graham-pos} that each $c_I^v \in H^*_S(pt)$
is a polynomial in $t$ with non-negative coefficients.
\end{proof}

\begin{thm}
\label[theorem]{thm:mainpos}
Let $p_I:=\iota^*\sigma_{v_I} \in H^*_S(\Pet)$ for some Coxeter element $v_I$,
and consider the multiplication in $H^*_S(\Pet)$,
\begin{align*}
p_I \cdot p_J = \sum_{K\subset\Delta} c_{I,J}^K p_K.
\end{align*}
The structure constants $c_{I,J}^K \in H^*_S(pt)$ are polynomials in $t$ with non-negative coefficients.
\end{thm}
\begin{proof}
By Graham's equivariant positivity theorem
\cite[Prop 2.2, Thm 3.2]{graham:positivity},
the structure constants $c_{u,v}^w$ in the expansion
\[
\sigma_u \cdot \sigma_v = \sum c_{u,v}^w \sigma_w \in H^*_T(G/B)
\]
are polynomials in the $T$-weights of $Lie(U)$ with non-negative coefficients.
Then
\begin{equation*}
p_I \cdot p_J  =  \iota^*\sigma_{v_I} \cdot \iota^*\sigma_{v_J} = \sum d_{u,v}^w \iota^*\sigma_w,
\end{equation*}
where $d_{u,v}^w$ is the image of $c_{u,v}^w$ under the restriction map $\X(T)\to \X(S)$;
in particular, $d_{u,v}^w$ is a polynomial in $t$ with non-negative coefficients.
The result now follows from \Cref{thm:hompos,thm:expansions},
since the classes $\iota^*\sigma_w$ expand into the classes $p_K$
with coefficients having the same positivity property as the $d_{u,v}^w$.
\end{proof}

\begin{rem}
\label{rmk:posrmk}
In the recent preprint \cite{goldin.gorbutt:PetSchubCalc}, Goldin and Gorbutt 
found a manifestly positive formula for the coefficients $c_{I,J}^K$, in Lie type A, and 
for a particular choice of the Coxeter elements $v_I$. While this paper was in preparation, 
a different combinatorial model, in the non-equivariant cohomology,
appeared in the preprint \cite{abe.horiguchi.kuwata.zeng:left-right} by Abe, Horiguchi,
Kuwata and Zeng. They also provide
a geometric proof of positivity (cf. Prop. 4.15 in {\em loc.~cit.}), which utilizes a 
`Giambelli formula', writing the classes $p_I$ in terms of 
products of pull-backs of the (effective) line bundles $GL_n \times^B \mathbb{C}_{-\omega_i}$ 
associated to the fundamental weights $\omega_i$. This argument 
should extend to arbitrary Lie type if one utilizes instead the more general
equivariant Giambelli formulae obtained by Drellich \cite{drellich:monk}, 
specialized to ordinary cohomology.
\end{rem}

\section{Stability Properties}
\label{sec:stable}
In this section, we utilize a common alternate construction of the Peterson variety
in order to prove a stability property of Peterson varieties.
For each finite-type Dynkin diagram $\Delta$,
we construct a variety $\Pet(\Delta)$ inside the flag manifold $\Fl(\Delta)$
which is isomorphic to the Peterson variety $\Pet$
corresponding to any group $G$ whose Dynkin diagram is $\Delta$.
The equality $\Pet(\Delta)=\Pet$ is well-known to experts; 
in \cref{sec:appendix}, we present a proof following Kostant \cite{kostant:flag}.

For $I\subset\Delta$,
we show that there is a natural inclusion $\Pet(I)\hookrightarrow\Pet(\Delta)$
identifying $\Pet(I)$ with the Peterson cell closure $\Pet_I$.
This implies that the fundamental classes $[\Pet_K]_S$ and the cohomology classes $p_K$
are stable for the inclusion $\Pet(I)\hookrightarrow\Pet(\Delta)$,
and that
the Peterson Schubert varieties of \cite{insko.tymoczko:intersection.theory}
are simply Peterson varieties corresponding to smaller groups.

\subsection{The Flag Manifold of a Dynkin diagram}
\label{sec:flagM}
Let $\Phi_\Delta$ (resp. $\Phi^+_\Delta$, $W_\Delta$) denote the root system (resp. positive roots, Weyl group)
corresponding to a finite-type Dynkin diagram $\Delta$.
Following \cite{tits:LAconstruction}, let $\mathfrak g_\Delta$ be the canonical complex semisimple Lie algebra associated to $\Delta$.
Recall that $\mathfrak g_\Delta$ comes with elements $\{e_\alpha,h_\alpha\}_{\alpha\in  \Phi_\Delta}$,
such that the $h_\alpha$ span a Cartan subalgebra $\mathfrak h_\Delta$ of $\mathfrak g_\Delta$,
and the $(e_\alpha)_{\alpha\in  \Phi_\Delta}$ form a Chevalley system for $(\mathfrak g_\Delta,\mathfrak h_\Delta)$; see \cite[Ch~7, \S2]{bourbaki:Lie79}.
We denote by $\mathfrak b_\Delta$ (resp. $\mathfrak b^-_\Delta$) the Borel subalgebra of $\mathfrak g_\Delta$ spanned by $\mathfrak h_\Delta$ and the set $\set{e_\alpha}{\alpha\in  \Phi_\Delta^+}$ (resp. $\set{e_\alpha}{\alpha\in  \Phi_\Delta^-}$).

We fix a connected Lie group $G$ with $Lie(G)=\mathfrak g_\Delta$. 
The adjoint action of $G$ on $\mathfrak{g}_\Delta$ induces an action on 
the Grassmannian $\mathrm{Gr}(\dim \mathfrak{b}_\Delta, \mathfrak{g}_\Delta)$. The orbit of 
$\mathfrak{b}_\Delta$ is closed, and it gives the 
\emph{flag variety} $\Fl(\Delta)$; see \cite[\S3.1]{chriss.ginzburg}. The Borel subalgebras of $\mathfrak{g}$ are conjugate under the adjoint action giving the following description of the flag variety: 
\begin{align}
\label{defnFlag}
\Fl(\Delta)=\set{\mathfrak b\subset\mathfrak g_\Delta}{\mathfrak b \text{ a Borel subalgebra of }\mathfrak g_\Delta}.
\end{align}
The stabilizer of $\mathfrak b_\Delta$ in $G$ is the Borel subgroup $B\subset G$ satisfying $\mathfrak b_\Delta=Lie(B)$,
hence we have the usual $G$-equivariant identification,
\begin{equation}
\label{eqn:flagVariety}
\varphi:G/B \overset\sim\to \Fl(\Delta).
\end{equation}
For $I\subset\Delta$,
the subalgebra of $\mathfrak g_\Delta$ spanned by $\{e_\alpha,h_\alpha\}_{\alpha\in\Phi_I}$
is precisely the Lie algebra $\mathfrak g_I$ associated to the Dynkin diagram $I$.
We have $\mathfrak h_I=\mathfrak h\cap\mathfrak g_I$ and 
$\mathfrak b_I=\mathfrak b_\Delta\cap\mathfrak g_I$.
Let
$T_I$, $B_I$, and $G_I$
be the connected subgroups of $G$ corresponding to 
$\mathfrak h_I$, $\mathfrak b_I$ and $\mathfrak g_I$ respectively.
The induced map $G_I/B_I\to G/B$ corresponds to an embedding $\Fl(I)\to\Fl(\Delta)$ via \cref{eqn:flagVariety}.
In \cref{eq:flaginclusion}, we give a characterization of this embedding
in terms of \cref{defnFlag}.

\begin{lemma}
\label{BorelLemma}
If $\mathfrak u\subset\mathfrak g_\Delta$ is a $|\Phi_\Delta^+|$-dimensional
subalgebra containing only nilpotent elements, then its normalizer
$N(\mathfrak u)=\set{x\in\mathfrak g}{ad(x)\mathfrak u\subset\mathfrak u}$
is a Borel subalgebra of $\mathfrak g_\Delta$.
\end{lemma}
\begin{proof}
Following \cite[p. 162, Cor 2]{bourbaki:Lie79},
every subalgebra $\mathfrak{u}$ containing only nilpotent elements is contained in some Borel subalgebra $\mathfrak b$,
and further, $\mathfrak u\subset[\mathfrak b,\mathfrak b]$ \cite[p. 91, Prop 5(b)]{bourbaki:Lie79}.
Comparing dimensions, we deduce that $\mathfrak u=[\mathfrak b,\mathfrak b]$,
and hence $N(\mathfrak u)=\mathfrak b$.
\end{proof}

Let $\mathfrak b_I'$ be any Borel subalgebra of $\mathfrak g_I$.
Observe that $\mathfrak v_I=\bigoplus\limits_{\alpha\in\Phi_\Delta^+\backslash\Phi_I^+}\mathfrak g_\alpha$ is $\mathfrak g_I$-stable, and
hence it is an ideal in the $|\Phi_\Delta^+|$-dimensional subalgebra $[\mathfrak b_I',\mathfrak b_I']\oplus\mathfrak v_I$.
By \cite[p. 71, Lemma 1]{bourbaki:Lie13},
we see that $[\mathfrak b_I',\mathfrak b_I']\oplus\mathfrak v_I$
is a $|\Phi_\Delta^+|$-dimensional subalgebra of $\mathfrak g_\Delta$ containing only nilpotent elements.
Following \cref{BorelLemma},
we see that $N([\mathfrak b_I',\mathfrak b_I']\oplus\mathfrak v_I)$
is a Borel subalgebra of $\mathfrak g_\Delta$.
Hence we have an embedding,
\begin{align}
\label{eq:flaginclusion}
i:\Fl(I)\to\Fl(\Delta), &&\mathfrak b_I'\mapsto N([\mathfrak b_I',\mathfrak b_I']\oplus\mathfrak v_I).
\end{align}

The embedding $i:\Fl(I)\to\Fl(\Delta)$ is $G_I$-equivariant,
and sends $\mathfrak b_I$ to $\mathfrak b_\Delta$.
It follows that under the identifications $\Fl(I)=G_I/B_I$ and $\Fl(\Delta)=G/B$ of \cref{eqn:flagVariety},
the map $i$ is precisely the map $G_I/B_I\to G/B$ induced by the inclusion $G_I\hookrightarrow G$; observe that $B_I= B \cap G_I$ follows from, e.g., \cite[\S 11.2, Corollary and Thm.~11.16 ]{borel:linear}.

We will say that a map of Lie groups $F: G_1 \to G_2$ {\em lifts} a Lie algebra map $f: \mathfrak{g}_1 \to \mathfrak{g}_2$ if $Lie(G_i) = \mathfrak{g}_i$ for $i=1,2$, and $f$ is the differential of $F$ at the identity.
\begin{rem}
The inclusion $i:\Fl(I)\to\Fl(\Delta)$ is $f$-equivariant for any map
$f:G'_I\to G$ lifting the inclusion $\mathfrak g_I\hookrightarrow\mathfrak g_\Delta$.
\end{rem}

\begin{lemma}
\label{schubStab}
Fix $w\in W_I$,
and let $\mathfrak b'_w=\mathfrak h_I\oplus\bigoplus_{\alpha\in\Phi_I^+}\mathfrak g_{w(\alpha)}$,
and $\mathfrak b_w=\mathfrak h_\Delta\oplus\bigoplus_{\alpha\in\Phi^+_\Delta}\mathfrak g_{w(\alpha)}$.
Consider the Schubert varieties
\begin{align*}
&&X^I_w=\overline{Ad(B_I)\mathfrak b'_w}\subset\Fl(I)&&\text{and}&& X_w=\overline{Ad(B)\mathfrak b_w}\subset\Fl(\Delta).
\end{align*}
Then $i(\mathfrak b'_w)=\mathfrak b_w$ and $i(X^I_w)=X_w$.
We view the $X_w^I$ as $B$-varieties via this identification.
Consider the Schubert classes
$\sigma_w\in H^*_T(\Fl(\Delta))$ and $\sigma^I_w\in H^*_T(\Fl(I))$.
We have
\begin{align*}
i_*[X^I_w]_T=[X_w]_T,&&i^*\sigma_w=\sigma^I_w.
\end{align*}
\end{lemma}
\begin{proof}
Since $w\in W_I$, we have $w(\Phi_\Delta^+\backslash\Phi_I^+)=\Phi_\Delta^+\backslash\Phi_I^+$,
and hence 
\begin{equation*}
\bigoplus\limits_{\alpha\in\Phi_\Delta^+}\mathfrak g_{w(\alpha)}=
\mathfrak v_I\oplus\bigoplus\limits_{\alpha\in\Phi_I^+}\mathfrak g_{w(\alpha)}
=[\mathfrak b'_w,\mathfrak b'_w]\oplus\mathfrak v_I.
\end{equation*}
It follows that $i(\mathfrak b'_w)=\mathfrak b_w$.
Next, since $B_I\subset B$, we have $i(X^I_w)\subset X_w$.
Further, both varieties are irreducible of dimension $l(w)$, hence they are equal.
Consequently, we have $i_*[X^I_w]_T=[X_w]_T$;
since the Schubert classes $\sigma_w$ (resp. $\sigma^I_w)$ are dual to the fundamental classes $[X_w]_T$ (resp. $[X^I_w]_T$),
we further obtain $i^*\sigma_w=\sigma^I_w$.
\end{proof}

\subsection{The Peterson Variety}\label{sec:petvarstab}
Given a Borel subalgebra $\mathfrak b\subset \mathfrak g_\Delta$,
let $\mathfrak h$ be a Cartan subalgebra of $\mathfrak b$,
let $\Phi_{\mathfrak h}$
denote the root system of $(\mathfrak g_\Delta,\mathfrak h)$,
and let $\Delta_{\mathfrak b} \subset \Phi_{\mathfrak h}$
be the set of simple roots for which $\mathfrak b$ is the Borel subalgebra corresponding to the positive roots.
We define
\begin{equation}
\label{def:Hb}
\mathcal H(\mathfrak b)=\mathfrak b\oplus\bigoplus\limits_{\alpha\in \Delta_{\mathfrak b}}\mathfrak g_{-\alpha,\mathfrak h}
\end{equation}
where $\mathfrak g_{\alpha,\mathfrak h}$ is the root space corresponding to $\alpha\in\mathfrak h^*$.

Observe that the subspace $\mathcal H(\mathfrak b)$ is independent of the choice of $\mathfrak h$.
Indeed, any two Cartan subgroups $\mathfrak h$ and $\mathfrak h'$ of $\mathfrak b$ are conjugate via an inner automorphism 
of $\mathfrak b$ \cite[Ch~7, \S3, Prop~5]{bourbaki:Lie79}.
Since $\mathcal H(\mathfrak b)$ is stable under the adjoint action of $\mathfrak b$,
the automorphism preserves $\mathcal H(\mathfrak b)$.  (Alternatively, 
$\mathcal H(\mathfrak b) = [\mathfrak{u}, \mathfrak{u}]^\perp$, where $\mathfrak{u}$
is the nilpotent radical of $\mathfrak{b}$, and $\perp$ is taken with respect to the Killing form.)

\begin{defn}
\label{altDefn}
Let $\n := \sum_{\alpha\in \Delta} e_\alpha$.
The {\em Peterson variety} $\Pet(\Delta)$ is defined by
\begin{align*}
\Pet(\Delta):=\set{\mathfrak b\in \Fl(\Delta)}{\n\in\mathcal H(\mathfrak b)}.
\end{align*}
We recall that $e$ is a regular nilpotent element of $\mathfrak g_\Delta$.
Under the $G$-equivariant isomorphism $G/B\overset\sim\to\Fl(\Delta)$ from \cref{eqn:flagVariety} we have
\begin{equation}\label{E:altPet}
\begin{split}
\Pet(\Delta) &= \set{gB\in G/B}{\n\in \mathcal H(Ad(g)\mathfrak b_\Delta)}\\
&= \set{gB\in G/B}{Ad(g^{-1})\n \in H(\mathfrak b_\Delta) = 
\mathcal \mathfrak b_\Delta\oplus\bigoplus\limits_{\alpha\in\Delta}\mathbb C e_{-\alpha}}.
\end{split}
\end{equation}
\end{defn}

Let $G$, $G_I$, and $\phi:G_I\to G$ be as in \cref{sec:flagM},
and let $S_I\subset T_I$ be the one-dimensional torus
corresponding to $h_I=\sum\limits_{\alpha\in\Phi_I^+}\alpha^\vee$.

\begin{prop}
\label{cor:PSisPet}
Consider the map $i:\Fl(I)\hookrightarrow\Fl(\Delta)$ from \cref{eq:flaginclusion}.
Then  $i(\Pet(I))=\Pet_I$, as algebraic varieties. Furthermore, $\Pet_I$ is also equal to the set theoretic 
intersection $\Pet(\Delta)\cap\Fl(I)$.
\end{prop}
\begin{proof} 
Let $e_I=\sum_{\alpha\in I}e_\alpha$ and $e_{\overline I}=\sum_{\alpha\in\Delta\backslash I}e_\alpha$,
so that $\n=e_I+e_{\overline I}$.
Recall that
\begin{align*}
\Pet(I)=\set{\mathfrak b_I'\in \Fl(I)}{e_I\in\mathcal H(\mathfrak b_I')}.
\end{align*}
Consider $\mathfrak b_I'\in\Fl(I)$,
and set $i(\mathfrak b_I')=\mathfrak b'$.
We see from \cref{eq:flaginclusion,def:Hb} that
\begin{align*}
\mathcal H(\mathfrak b_I')\oplus\mathfrak v_I\subset \mathcal H(\mathfrak b').
\end{align*}
Suppose $\mathfrak b_I'\in\Pet(I)$.
We have $e_{\overline I}\in\mathfrak v_I$,
and hence
\begin{align*}
e_I\in\mathcal H(\mathfrak b_I')\implies
\n=e_I+e_{\overline I}\in\mathcal H(\mathfrak b_I')\oplus\mathfrak v_I\subset\mathcal H(\mathfrak b')
\implies \mathfrak b'\in\Pet(\Delta).
\end{align*}
We deduce that $i(\Pet(I))\subset \Pet(\Delta)$.

 Using the natural basis $\{e_\alpha,h_\alpha\}_{\alpha\in\Phi_\Delta}$ for $\mathfrak{g}_\Delta$, and its sub-basis of $\mathfrak{g}_I$,
consider the 
$\mathfrak g_I$-equivariant projection $\mathrm{pr}:\mathfrak g_\Delta\to\mathfrak g_I$ defined by:
\begin{align*}
\mathrm{pr}(e_\alpha)=
\begin{cases}
e_\alpha&\text{ if }\alpha\in\Phi_I,\\
0       &\text{otherwise.}\\
\end{cases}
&&
\mathrm{pr}(h_\alpha)=
\begin{cases}
h_\alpha&\text{ if }\alpha\in\Phi_I,\\
0       &\text{otherwise.}\\
\end{cases}
\end{align*}
Now, suppose $\mathfrak b'\in\Pet(\Delta)$.
Then $\n\in\mathcal H(\mathfrak b')$, and hence
\begin{align*}
e_I=\mathrm{pr}(\n)\in\mathrm{pr}(\mathcal H(\mathfrak b'))=\mathcal H(\mathfrak b_I').
\end{align*}
It follows that $\Fl(I)\cap\Pet(\Delta)=\Pet(I)$.
The equality $\Pet(I)=\Pet_I$ is a consequence of the observation that $\Fl(I)=X_{w_I}$,
and the irreducibility of $\Pet(I)$; see \cref{prop:pet.irred}.
\end{proof}

We will denote by $j:\Pet(I)\to\Pet(\Delta)$ the inclusion induced by restricting $i$ to $\Pet(I)$.
{In order to discuss stability for Peterson classes, we first need to
construct algebra homomorphisms $H^*_S(\Fl(\Delta){;\IQ}) \to H^*_{S_I}(\Fl(I){;\IQ})$,
compatible with restrictions to Peterson subvarieties.
To this end, we replace $S_I$ and $S$ by a $\mathbb C^*$
`parametrizing' {(not necessarily injectively)} these tori
via the
the defining cocharacters $h_I: \mathbb C^* \to S_I$ and $h: \mathbb C^* \to S$.
This $\mathbb C^*$ acts on $\Fl(I)$, respectively on $\Fl(\Delta)$,
via its image $S_I \subset T_I$ and $S \subset T$.
The embedding $\mathfrak g_I\to\mathfrak g$ is $\mathbb C^*$-equivariant,
and hence so is the embedding $\Fl(I)\to\Fl(\Delta)$ described in \eqref{eq:flaginclusion}.
These facts are summarized in the diagram below.
The question marks signify that a map may not exist;
see \Cref{rmk:nomap} below.
\begin{center}
\begin{tikzcd}
&\mathbb C^*\arrow[dl,"h_I",swap]\arrow[dr,"h"]&\\
S_I\arrow[rr,dashed,"??"]&&S\\
\mathcal Fl(I)\arrow[rr,hookrightarrow]\arrow[loop, looseness=5]&&\mathcal Fl(\Delta)\arrow[loop, looseness=5]
\end{tikzcd}
\end{center}

The cocharacter $h$ induces {an isomorphism $Lie(\mathbb C^*)\to Lie(S)$,
and hence} a ring isomorphism $H^*_S(pt; \mathbb{Q}) \to H^*_{\mathbb{C}^*}(pt; \mathbb{Q})$.
(In general the corresponding map over integer coefficients, 
$H^*_{S}(pt;\mathbb Z)\to H^*_{\mathbb{C}^*}(pt;\IZ)$, may not be an isomorphism.)
The identity map $\Fl(\Delta) \to \Fl(\Delta)$ is equivariant with respect to the cocharacter
$h: \mathbb{C}^* \to S$, therefore by functoriality we have induced isomorphisms
$H^*_{S}(\Fl(\Delta); \mathbb{Q}) \to H^*_{\mathbb{C}^*}(\Fl(\Delta); \mathbb{Q})$ and 
$H_*^{\mathbb{C}^*}(\Fl(\Delta);\mathbb{Q}) \to H^S_{*}(\Fl(\Delta); \mathbb{Q})$. 
Further, since $\Pet(\Delta)$ is $S$-stable,
it inherits a $\mathbb C^*$-action through $h$, giving isomorphisms
\begin{align*}
H^*_{S}(\Pet(\Delta);\mathbb Q)\overset\sim\to H^*_{\mathbb C^*}(\Pet(\Delta);\mathbb Q)&&
\text{and}&&
H_*^{\mathbb C^*}(\Pet(\Delta);\mathbb Q)\overset\sim\to H_*^{S}(\Pet(\Delta);\mathbb Q) \/.
\end{align*} All these isomorphisms are natural with respect to the closed embedding $\Pet(\Delta) \subset \Fl(\Delta)$.
A similar discussion 
for the cocharacter $h_I$ yields isomorphisms 
\begin{align*}
H^*_{S_I}(\Pet(I);\mathbb Q)\overset\sim\to H^*_{\mathbb C^*}(\Pet(I);\mathbb Q)&&
\text{and}&&
H_*^{\mathbb{C}^*}(\Pet(I);\mathbb Q)\overset\sim\to H_*^{S_I}(\Pet(I);\mathbb Q) \/,
\end{align*} natural with respect to $\Pet(I) \subset \Fl(I)$.} 
Consequently, the $\mathbb C^*$-equivariant inclusion $j:\Pet(I)\to\Pet(\Delta)$ yields
a pullback map,
\begin{align*}H^*_S(\Fl(\Delta);\mathbb Q)\to H^*_{S_I}(\Fl(I);\mathbb Q)\end{align*}
compatible with the algebra isomorphism $H^*_S(pt; \mathbb Q)\to H^*_{S_I}(pt;\IQ)$,
and we obtain a commutative diagram,
\begin{equation}
\label{commQ}
\begin{tikzcd}
H^*_S(\Pet(\Delta);\mathbb Q)\arrow[r,dashed,"j^*"]\arrow[d,"\cong"] & H^*_{S_I}(\Pet(I);\IQ)\arrow[d,"\cong"]\\
H^*_{\mathbb C^*}(\Pet(\Delta);\mathbb Q)\arrow[r]                               & H^*_{\mathbb C^*}(\Pet(I);\IQ).
\end{tikzcd}
\end{equation}
In a similar fashion, we also obtain a pushforward $j_*:H_*^{S_I}(\Pet(I);\IQ)\to H_*^S(\Pet(\Delta);\IQ)$.

The following is the main result of this section.
\begin{thm}
\label[theorem]{prop:naturality}
Consider the map $i:\Fl(I)\hookrightarrow\Fl(\Delta)$ from \cref{eq:flaginclusion}.
\begin{enumerate}[label=(\alph*)]
\item
For $J\subset I$, we have $i_*[\Pet_J]_{S_I}=[\Pet_J]_S$ in $H_*^{S}(\Fl(\Delta);\mathbb{Q})$.
\item 
Let $j^*:H^*_S(\Pet(\Delta);\IQ)\to H^*_{S_I}(\Pet(I);\IQ)$ denote the pullback
induced from the inclusion $\Pet(I)\hookrightarrow \Pet(\Delta)$.
For $K\subset\Delta$, we have
$$j^*p_K=\begin{cases}p_K&\text{if }K\subset I,\\0&\text{otherwise}.\end{cases}$$
\end{enumerate}
In the non-equivariant case, the equalities in (a) and (b) hold with integral coefficients.
\end{thm}
\begin{proof}
For $J\subset I\subset\Delta$, the inclusions $\Fl(J)\overset{i'}\hookrightarrow\Fl(I)\overset i\hookrightarrow\Fl(\Delta)$
are $\mathbb C^*$-equivariant for the action given by the cocharacters
$h_J,h_I$ and $h$,  respectively. By \Cref{cor:PSisPet},
we have $i'(\Pet(J))=\Pet_J\subset\Pet(I)$ and $i(i'(\Pet(J)))=\Pet_J\subset\Pet(\Delta)$,
and consequently $[\Pet_J]_{\mathbb C^*}=i_*(i'_*([\Pet(J)]_{\mathbb{C^*}}))=i_*([\Pet_J]_{\mathbb{C}^*})$
in $H_*^{\mathbb{C}^*}(\Fl(\Delta))$. Then part (a) follows because the $\mathbb{C}^*$-equivariance may be replaced by the 
$S_I$, respectively $S$-equivariance, as explained above.
Part (b) follows from \cref{schubStab} and the commutativity of the diagram,
\begin{center}
\begin{tikzcd}
\Pet(I)\arrow[r,"j",hook]\arrow[d,hook,"\iota"]&\Pet(\Delta)\arrow[d,hook,"\iota"]\\
\Fl(I)\arrow[r,"i",hook] &\Fl(\Delta)
\end{tikzcd}
\end{center}
utilizing again that all maps are $\mathbb{C}^*$-equivariant.

In the non-equivariant case, all (co)homology morphisms are defined over $\mathbb{Z}$, and the
classes $[\Pet_I]$ and $p_I$ are integral, by their definition. This finishes the proof.
\end{proof}

\begin{rem}
\label{rmk:nomap}
The 
reader may wonder whether an
algebra map $H^*_S(\Fl(\Delta)) \to H^*_{S_I}(\Fl(I))$
may be directly constructed from the inclusion $i: \Fl(I) \to \Fl(\Delta)$,
equivariant with respect to a map $\varphi_I: S_I \to S$. 
The requirement that $i$ is $\varphi_I$-equivariant implies that the differential 
$d\varphi_I: Lie(S_I)\to Lie(S)$ must send $h_I\mapsto h$.
(Note that this is {\em not} the restriction of the natural map $Lie(T_I) \hookrightarrow Lie(T)$.)
The existence of a lift $S_I \to S$ of this Lie algebra map cannot be guaranteed.
For instance, consider the inclusion $G_I:=SL_3 \subset G:=SL_4$
given by the natural embedding of Dynkin diagrams $A_2 \subset A_3$.
The tori $S_I$ and $S$ are the images of cocharacters 
\[ h_I(z) = \begin{pmatrix} z^2 & 0 & 0 \\ 0 & 1 & 0 \\ 0 & 0 & z^{-2} \end{pmatrix} \quad \textrm{ and } \quad 
h(z) = \begin{pmatrix} z^3 & 0 & 0 & 0  \\ 0 & z & 0 & 0 \\ 0 & 0 & z^{-1} & 0 \\ 0 & 0 & 0 & z^{-3} \end{pmatrix} \/, \]
respectively. In this case there is no group homomorphism $\varphi_I:S_I \to S$  satisfying $\varphi_I(h_I(z))=h(z)$.
\end{rem}

\begin{cor}\label{cor:stabZ} Let $I \subset \Delta$ and assume that 
the map $Lie(S_I)\to Lie(S)$ sending $h_I \mapsto h$ lifts to a map $\varphi_I: S_I\to S$.
Then the push-forward and pull-back maps
\[ j^*: H^*_S(\Pet(\Delta)) \to H^*_{S_I}(\Pet(I)) \textrm{ and } j_*: H_*^{S_I}(\Pet(I)) \to 
H_*^{S}(\Pet(\Delta)) \]
may be defined with $\mathbb{Z}$ coefficients. In particular, 
the statements in \Cref{prop:naturality} also hold over $\mathbb{Z}$.
\end{cor}
\begin{proof} The claim follows because the $G_I$-equivariant map $i: \Fl(I) \to \Fl(\Delta)$ 
from \eqref{eq:flaginclusion} restricts to the $\varphi_I$-equivariant map 
$j: \Pet(I) \to \Pet(\Delta)$. Then $j_*$ and $j_*$ may be defined over $\mathbb{Z}$.
\end{proof}

\begin{rem}
The results of this section can be extended to the case of reductive groups $G$ and a one-dimensional torus $S\subset T$ satisfying $\alpha|S=\beta|S$ for all simple roots $\alpha,\beta$.
For $G$ semisimple, there is a unique $S\subset T$ satisfying this condition. 
For an arbitrary reductive group $G$, this may not determine $S$ uniquely.

It is common in the literature on type A Peterson varieties to use the group $G=GL_n$ 
and the one-dimensional torus $S=diag(z^n,z^{n-1}, \ldots, z)$.
In this case, we have an identification between the one dimensional 
subtori of $GL_n$ and $GL_{n+1}$ given by 
$diag(z^n,z^{n-1},\ldots, z)\mapsto diag(z^{n+1},z^n,\ldots, z)$. 
Then  
the diagram in \eqref{commQ}, and hence the statements in 
\Cref{prop:naturality}, hold over $\mathbb{Z}$.
\end{rem}

\section{Intersection multiplicities}
\label{sec:intmult}
Different choices of Coxeter elements $v_I$ lead to different bases $\{ p_I = \iota^*\sigma_{v_I} \}$ for 
$H^*_S(\Pet; \mathbb Q)$.
By \Cref{thm:duals}, the transition matrix between 
two such bases $\{ p_I \} $ and $\{ p'_I \}$ is diagonal,
with entries given by ratios
\[
\frac{m(v_I)}{m(v'_I)} = \frac{\langle p_I , [\Pet_I]_S \rangle}{\langle p_I' , [\Pet_I]_S \rangle}.
\]
It is natural to ask whether there are choices for the Coxeter elements $v_I$ for which
$m(v_I) = 1$,
and more generally, to ask for formulae for the $m(v_I)$.
In \cref{rootHts}, we give a formula for $m(v_I)$ in terms of
the localization of the Schubert variety $X^{v_I}$ at the point $w_I$,
and in \cref{ClassicalFormulae}, we use this formula to compute $m(v_I)$ for certain Coxeter elements $v_I$.
\Cref{ClassicalFormulae} settles Question 1 of \cite{insko.tymoczko:intersection.theory} for all classical types.
As a further application of \cref{rootHts}(b),
we show in \cref{Acounter} that not all choices of $v_I$ lead to $m(v_I)=1$ in type A,
and in \cref{BCcounter} that for $I\in\{B_2, C_2\}$,
there is no Coxeter element $v_I$ for which $m(v_I)=1$.

\subsection{The Exponents of a Dynkin diagram}
\label{carter}
Let $\Delta$ be a Dynkin diagram with $n$ nodes. 
The exponents $m_1,\cdots,m_n$ of $\Delta$ are fundamental invariants, appearing in many contexts.
We will utilize the following two characterizations found in \cite[Ch. 10]{carter:simple}; see also \cite{kostant:principal}:
\begin{enumerate}
\item
Let $\mathfrak g$ be the Lie algebra with Dynkin diagram $\Delta$,
and let $\{e,f,h\}$ be an $\mathfrak{sl}_2$-triple in $\mathfrak g$,
such that $e$ is a regular nilpotent element in $\mathfrak g$;  see \cite{morozov,mcgovern.collingwood:nilpotent.orbits}.
The $\mathfrak{sl}_2$-decomposition of $\mathfrak g$ is precisely $\oplus V(2m_i)$,
where $V(k)$ denotes the irreducible finite dimensional $\mathfrak{sl}_2$-representation with highest weight $k$.
\item
Let $a_i$ be the number of roots of height $i$ in $\Phi_\Delta^+$.
Then $(a_1,\cdots, a_k)$ is a partition, and the conjugate partition is precisely $(m_1,\cdots,m_n)$.
\end{enumerate}
\begin{table}[ht]
\begin{tabular}{ |c|c||c|c| } 
 \hline
Diagram       & Exponents               & Diagram & Exponents\\ \hline
$ A_n$        & $1, 2,\cdots, n$        & $F_4$   & $1,5,7,11$\\
$ B_n$, $C_n$ & $1, 3,\cdots,2n-1$      & $E_6$   & $1, 4, 5, 7, 8, 11$\\
$D_n$         & $1, 3,\cdots,2n-3, n-1$ & $E_7$   & $1,5,7,9,11,13,17$\\
$G_2$         & $1, 5$                  & $E_8$   & $1, 7, 11, 13, 17, 19, 23, 29$\\ \hline
\end{tabular}
\caption{The exponents of Dynkin Diagrams; see \cite[Ch. 10]{carter:simple}.}
\label{table:exponents}
\end{table}
Throughout this section, we will denote by $m_1,\cdots,m_n$, the exponents of $\Delta$.

\begin{lemma}
\label{wtsADelta}
The weights for the $S$-action on $Lie(G^e)$ are precisely $m_1 t,\cdots, m_n t$.
\end{lemma}
\begin{proof}
Recall that $S\subset T$ corresponds to the cocharacter $h$
satisfying $\alpha(h)=2$ for all $\alpha\in\Delta$, and that $[h,e]=2e$.
Identifying $\X(S)$ as a lattice in $Lie(S)^*$, we view $t$ as an element of $Lie(S)^*$.
Let $\varpi\in Lie(S)^*$ be the fundamental weight dual to $h$, i.e., given by $\varpi(h)=1$.
Comparing the weights of the $h$-action and $S$-action on $e$, we deduce that $t=2\varpi$.

Consider now an $\mathfrak{sl}_2$-triple in $\mathfrak g$,
with $e$ (resp. $h$) as the nilpositive (resp. neutral) element.
Since $e$ is a principal nilpotent element of $\mathfrak g$,
the decomposition of $\mathfrak g$ as an $\mathfrak{sl}_2$-representation is given by
$
\mathfrak g=\oplus V(2m_i\varpi)=\oplus V(m_i t).
$
Now, simply observe that
\begin{align*}
Lie(G^e)=\set{x\in Lie(U)}{[e,x]=0}=\ker(ad(e))
\end{align*}
is spanned by the highest weight vectors in $\mathfrak g$,
whose weights are precisely $m_1t,\cdots,m_n t$.
\end{proof}

\begin{lemma}
\label{eulerClassTangent}
The $S$-equivariant Euler class of the tangent space $T_{w_I}\Fl(I)$ is $\left(\prod m_i!\right)t^N$, where $N=\dim\Fl(I)$.
\end{lemma}
\begin{proof}
Observe that the map $\X(T)\to \X(S)$ is given by $\alpha\mapsto t$, for all $\alpha\in\Delta$.
Consequently, the $T$-weight space $\mathfrak g_\alpha$, for $\alpha\in\Phi^+_I$, is an $S$-weight space of weight $ht(\alpha) t$.
The tangent space at $w_I$ admits a $T$-decomposition,
\begin{equation*}
T_{w_I}(G/B)=\bigoplus\limits_{\alpha\in\Phi_I^+}\mathfrak g_\alpha \/;
\end{equation*}
hence the $S$-equivariant Euler class of $T_{w_I}(G/B)$ is
$t^{a_1}(2t)^{a_2}\ldots (kt)^{a_k}$,
where $a_i$ is the number of roots of height $i$ in $\Phi_I^+$.
Following \cref{carter}, the partition $(a_1,\ldots,a_k)$ is conjugate to $(m_1,\ldots,m_n)$;
consequently, the $S$-equivariant Euler class of $T_{w_I}(G/B)$ is precisely $m_1!\,m_2!\cdots m_n!\,t^N$.
\end{proof}

We are now ready to calculate the multiplicities $m(v_I)$ using the map in cohomology obtained by restricting to the fixed point set.

\begin{prop}
\label{rootHts}
Let $\iota_{w}^*: H_S^*(G/B) \rightarrow H_S^*(w)$ be the map induced by the inclusion $wB/B\hookrightarrow G/B$.
Define $b\in \mathbb Z$ by 
$\iota_{w_I}^*\sigma_{v_I}=b t^n$. 
\begin{enumerate}[label=(\alph*)]
\item
We have
$m(v_I)=\dfrac b{m_1\cdot \ldots \cdot m_n}$.
\item
Suppose $X^{v_I}$ is smooth at $w_I$.
Let $\{\beta_1,\cdots,\beta_n\}=\set{\alpha\in\Phi_I^+}{s_\alpha \not\leq v_Iw_I}$.
Then
\begin{equation*}
m(v_I)=\frac{ht(\beta_1)\cdot \ldots \cdot ht(\beta_n)}{m_1 \cdot \ldots \cdot m_n}.
\end{equation*}
\end{enumerate}
\end{prop}
\begin{proof}
Recall from  \cref{intDual} that 
\begin{equation}
\label{intersectionEqn}
\sigma_{v_I}\cup \eta_I  = m(v_I) \tau_{w_I},
\end{equation}
where  $\eta_I$ and $\tau_{w_I}$ are Poincar\'e dual to $[\Pet_I]_S$ and $[w_I]_S$, respectively, in $H_S^*(G/B)$.

We restrict both sides to $w_I$ under the map $\iota_{w_I}^*: H_S^*(G/B) \rightarrow H_S^*(w_I)$.
By \cref{prop:naturality}, we may assume $\Delta=I$,
so that the tangent space $T_{w_I}\Pet_I=Lie(G^e)$ has $S$-weights $m_1t,\cdots,m_nt$ as described in \cref{wtsADelta}.
Following \cref{prop:paving,eulerClassTangent},
we see that the $S$-equivariant Euler class at $w_I$ of the normal bundle of $\Pet$ is 
$\dfrac{m_1!m_2! \cdot \ldots \cdot m_n!t^N}{m_1m_2 \cdot \ldots \cdot m_n t^n}$.
Applying $\iota^*_{w_I}$ to both sides of \cref{intersectionEqn} yields
\begin{align*}
\dfrac{m_1!m_2!\cdot \ldots \cdot m_n!t^N}{m_1m_2 \cdot \ldots \cdot m_n t^n}\,\iota^*_{w_I}\sigma_{v_I}=m(v_I)\iota^*_{w_I}\tau_{w_I}.
\end{align*}
Using \cref{eulerClassTangent}, we have $\iota^*_{w_I}\tau_{w_I}=m_1!m_2!\cdots m_n!t^N$, and part (a) follows.

For part (b), 
since $X^{v_I}$ is smooth at $w_I$, the normal space of $X^{v_I}$ at $w_I$ is spanned by
\begin{align*}
\set{\mathfrak g_\alpha}{\alpha\in\Phi^+_I,\,s_\alpha w_I\not\geq v_I}=
\set{\mathfrak g_\alpha}{\alpha\in\Phi^+_I,\,s_\alpha \not\leq v_Iw_I};
\end{align*}
see \cite[Cor 12.1.10]{kumar:book}.
Part (b) now follows from (a), along with the observation that
the map $\X(T)\to \X(S)$ is given by $\beta\mapsto ht(\beta)t$.
\end{proof}

\begin{example}
\label{Acounter}
Let $I=A_3$, and $v_I=s_1s_3s_2$.
Then $m(v_I)=2$.
\end{example}

\begin{example}
\label{BCcounter}
For $I\in\{B_2,C_2\}$, we have $m(v_I)=2$ for every Coxeter element $v_I$.
\end{example}

In \cite[Question 1]{insko.tymoczko:intersection.theory}, Insko and Tymoczko 
conjecture that $m(v_I)=1$ for certain Coxeter elements, when $I$ is contained in some sub-diagram of type $A$, 
and that $m(v_I)=2$ otherwise.
As an application of \cref{rootHts}, we compute $m(v_I)$ for one Coxeter element in each Dynkin diagram; 
this formula proves their conjecture in type A, and disproves it in other cases.

%
\begin{thm}
\label{ClassicalFormulae}
\label{cor:connComp}
\begin{enumerate}[label=(\alph*),leftmargin=*]
\item
Let $I$ be a connected Dynkin diagram with the standard labelling
 (see \cite{bourbaki:Lie46}), and set $v_I=s_1s_2\cdots s_n$.
Then,
\begin{align*}
m(v_I)=
\begin{cases}
1      &\text{if }I=A_n,\\
2^{n-1}&\text{if }I=B_n, C_n,\\
2^{n-2}&\text{if }I=D_n,\\
72=2^3 \cdot 3^2     &\text{if }I=E_6,
\end{cases}
&&
m(v_I)=
\begin{cases}
864=2^5 \cdot 3^3     &\text{if }I=E_7,\\
51840=2^7 \cdot 3^4 \cdot 5  &\text{if }I=E_8,\\
48=2^4 \cdot 3     &\text{if }I=F_4,\\
6 = 2 \cdot 3     &\text{if }I=G_2.
\end{cases}
\end{align*}
\item
Let $I_1,\cdots,I_k$ be the connected components of a Dynkin diagram $I$,
and let $v_1,\cdots,v_k$ be Coxeter elements for $I_1,\cdots,I_k$ respectively.
Then $v:=v_1\cdots v_k$ is a Coxeter element for $I$, and $m(v)=\prod\limits_{j=1}^km(v_j)$.
\end{enumerate}
\end{thm}
\begin{proof}
[Proof of \cref{ClassicalFormulae}]
If $I$ is a diagram of classical type, the variety $X^{v_I}$ is smooth at $w_I$,
cf. \cite[Thm 3]{insko.tymoczko:intersection.theory}.
Consequently, we can use \cref{rootHts}(b) to compute $m(v_I)$.
We show the details of the calculations in \cref{classicalCalculations}.
For the exceptional cases, a computer calculation suffices:
we use the localization formula (cf. \cite{AJS,billey:kostant}) to compute $\iota^*_{w_I}\sigma_{v_I}$, and apply \cref{rootHts}(a).

Following \cref{prop:naturality}, we may assume $\Delta=I$.
The integer $m(v)$ is the multiplicity of the intersection of $X^v$ with $\Pet$.
We have
\begin{align*}
\Fl(I)=\prod\limits\Fl(I_j),&& X^v=\prod X^{v_j},&& \Pet(I)=\prod\Pet(I_j),
\end{align*}
and hence the multiplicity $m(v)$ is the product of the multiplicities $m(v_j)$.
\end{proof}

\begin{rem}
We conjecture for all Coxeter elements $v_I$
a type-independent formula for the intersection multiplicity, namely:
\begin{equation}
\label{generalFormula}
m(v_I)=\dfrac{|\mathcal R(v_I)||W_I|}{|I|! \det(C_I)}=\mathcal R(v_I)\prod_{\alpha\in I} a_\alpha.
\end{equation}
Here $\mathcal R(v_I)$ is the set of reduced expressions for $v_I$,
$W_I$ is the Weyl group of $I$,
$C_I$ is the Cartan matrix of the Dynkin diagram $I$,
and the integers $a_\alpha$ are the coefficients of the highest root $\theta_I=\sum_{\alpha\in I} a_\alpha \alpha$ of $I$.
The second equality follows from \cite[p.~297]{bourbaki:Lie46}.
For the Coxeter elements in \cref{ClassicalFormulae},
we have verified this formula with type-by-type calculations.
 For a type-independent proof, see \cite{goldin.singh}.
\end{rem}

\begin{rem}
The formula $m(v_I)=1$ for $I=A_n$ was first obtained by Insko in \cite{insko:schubert},
who proved that the scheme-theoretic intersection $X^{v_I}\cap\Pet_I$ is reduced.
\end{rem}

\begin{cor}
\label{ITconj}
Suppose $I$ is contained in some sub-diagram $J$ of type A,
and let $v$ be the Coxeter element of $I$ obtained by multiplying the simple reflections in increasing order (for the standard type A labelling of nodes in $J$).
Then $m(v)=1$.
\end{cor}
\begin{proof}
Observe that each connected component $I_j\subset I$ is of type A.
Let $v_j$ be the Coxeter element of $I_j$ obtained by multiplying the simple reflections in increasing order
(for the standard type A labelling of nodes in $I_j$), so that $v=\prod v_j$.
Following \cref{ClassicalFormulae}, we have $m(v_j)=1$, and $m(v)=\prod m(v_j)=1$.
\end{proof}

\appendix
\section{Two Definitions of the Peterson Variety} \label{sec:appendix} 
In this section we recall the affine paving of the Peterson variety (\cref{sec:paving}),
and we show in  \cref{equalDefn} that the two definitions of the Peterson variety,
\begin{equation}\label{2defns}
\begin{split}
\quad \quad \Pet  & := \overline{G^e. w_0B}\quad \/, \\
\quad \quad \Pet(\Delta)& := \set{gB\in G/B}{Ad(g^{-1})e \in Lie(B)\oplus\bigoplus\limits_{\alpha\in\Delta}\mathbb C e_{-\alpha}} \/,
\end{split}
\end{equation}
agree. These results are well-known to experts,
but either some statements are only implicitly present in the literature,
or we present slightly different proofs. 
A key point is the 
irreducibility of $\Pet(\Delta)$, which we prove utilizing  
results of Kostant \cite{kostant:flag}.  
We also present in \cref{rem:irred} an alternate proof
following \cite{anderson.tymoczko,precup,abe.fujita.zeng}, 
as explained to us by B\u{a}libanu. Their arguments extend to
the wider setting of regular Hessenberg varieties.

\subsection{Paving by Affines}
\label{sec:paving}
For $I\subset \Delta$, let $\Pet(I)^\circ$ be 
the \emph{Peterson cell},
$
\Pet(I)^\circ=\Pet(I)\backslash\cup_{J\subsetneq I} \Pet(J).
$
Let $U_I$ be the unipotent Lie group corresponding to the Dynkin diagram $I$, and let $A_I$ be the centralizer of $e_I$ in $U_I$; see \cite{tits:LAconstruction}.

The following proposition was proved in various cases by Tymoczko 
 \cite[Theorem 4.3]{tymoczko:paving} and B\u{a}libanu \cite[Section 6]{balibanu:peterson}.
Following the exposition in \cite{balibanu:peterson},
we recall the main steps in the proof.
\begin{prop}
\label{prop:paving}
\begin{enumerate}[label=(\alph*),leftmargin=*]
\item
The group $A_I$ acts transitively and faithfully on $\Pet(I)^\circ$.
\item
$\Pet(\Delta)=\bigsqcup\limits_{I\subset \Delta} \Pet(I)^\circ$ is a paving by affines.
\item
The intersection $\Pet(\Delta)\cap X_w^\circ$ is nonempty if and only if $w=w_I$ for some subset $I\subset \Delta$.
\end{enumerate}
\end{prop}
\begin{proof}
Following \cite[Prop~6.3]{balibanu:peterson}, we have $\Pet(I)^\circ=A_Iw_IB/B$, i.e., $A_I$ acts transitively on $\Pet(I)^\circ$.
Further, $U_I$ acts faithfully on the Schubert cell $X_{w_I}^\circ$,
and hence the action of $A_I\subset U_I$ is faithful at the point $w_I$.
Next, observe that $\Pet(I)^\circ$ is a principal space for $A_I$, hence is an affine space.
Finally, the observation $\Pet(I)^\circ\subset X_{w_I}^\circ$, along with part (b) implies that $\Pet(\Delta)\cap X_w^\circ$ is empty unless $w=w_I$ for some $I$.
\end{proof}

\subsection{Equivalence of two definitions of the Peterson variety}

\begin{lemma}
{\rm(\cite{kostant:flag})}
\label{lem:irredOp}
The variety $\Pet(\Delta)$ is locally irreducible at the point $1B$.
\end{lemma}
\begin{proof}
Let $U^-$ be the unipotent radical of the opposite Borel subgroup $B^-$, and let
$\mathcal N$ be the variety of nilpotent elements in $\mathfrak g$.
Consider the map $\eta:U^-\to\mathfrak g$ given by $u\mapsto Ad(u^{-1})\n$.
Following \cite[Thm. 17]{kostant:flag}, the map $\eta$ induces an isomorphism,
\begin{align*}
\eta:U^-\overset\sim\to(\n+\mathfrak b^-)\cap\mathcal N,
\end{align*}
where $\mathfrak b^-=Lie(B^-)$.
Recall that the $U^-$-orbit of $1B$ is an open set (namely the opposite Schubert cell) in 
$G/B$.
 Hence $Z:=\Pet(\Delta)\cap U^-B/B$ is an open neighborhood of $1B$ in $\Pet(\Delta)$,
and it suffices to show that $Z$ is irreducible. 
Let $\mathfrak f=\bigoplus\limits_{\alpha\in \Delta}\mathfrak g_{-\alpha}$,
and let $\mathcal N^{reg}$ denote the set of regular nilpotent elements in $\mathfrak g$.
Since the intersection $U^- \cap B$ is trivial, it follows that 
$Z=\eta^{-1}(\mathfrak b\oplus\mathfrak f)$. Then
\begin{align*}
Z&=\eta^{-1}(\mathfrak b\oplus\mathfrak f)\cong(\mathfrak b\oplus\mathfrak f)\cap(\n+\mathfrak b^-)\cap\mathcal N\\
&=(\n+\mathfrak h+\mathfrak f)\cap\mathcal N=(\n+\mathfrak h+\mathfrak f)\cap\mathcal N^{reg},
\end{align*}
where the last equality is from \cite[\S3.2]{kostant:flag}.
The result now follows from the irreducibility of $(\n+\mathfrak h+\mathfrak f)\cap\mathcal N^{reg}$,
cf. \cite[Thm. 6]{kostant:flag}.
\end{proof}

\begin{lemma}
\label[lemma]{prop:pet.irred}
The variety $\Pet(\Delta)$ is irreducible.
\end{lemma}
\begin{proof}
Let $Y$ be an irreducible component of $\Pet(\Delta)$.
Recall
from \cref{sec:paving} 
that the (connected) group $A_\Delta=Stab_U(e)$
acts on $\Pet(\Delta)$, hence it acts on $Y$.
Since $Y$ is a closed (hence projective) subvariety of $G/B$
and since 
$A_\Delta$ is solvable, $Y$ admits an $A_\Delta$-fixed point by \cite[Thm.10.4]{borel:linear}.
This point must necessarily be $1B$, as this is the unique $A_\Delta$-fixed point in $G/B$.
In other words, every irreducible component of $\Pet(\Delta)$ contains $1B$;
the irreducibility of $\Pet(\Delta)$ now follows from the local irreducibility of $\Pet(\Delta)$ at $1B$; see \Cref{lem:irredOp}.
\end{proof}

\begin{prop}
\label{equalDefn}
The two definitions of the Peterson variety in \cref{2defns} agree, i.e., $\Pet=\Pet(\Delta)$.
\end{prop}
\begin{proof}
Observe that $G^e=A_\Delta\times Z(G)$,
where $Z(G)$ is the center of $G$,
cf. \cite[p.~9]{kostant:flag}.
Since $Z(G)\subset B$, we have an equality $A_\Delta. w_0B=G^e. w_0B \subset \Pet \cap \Pet(\Delta)$.
It follows from \cref{prop:paving} that $G^e . w_0B$ is an open subset of $\Pet(\Delta)$,
and since $\Pet(\Delta)$ is irreducible by \cref{prop:pet.irred}, $\Pet=\Pet(\Delta)$.
\end{proof}

\begin{rem}\label{rem:irred}
We recall an alternate proof of the irreducibility of the variety $\Pet(\Delta)$,
following Precup \cite[Cor 14]{precup} and \cite[Lemma 7.1]{anderson.tymoczko},
as explained to us by B{\u a}libanu.

Let $H=\mathfrak b\oplus(\oplus_{\alpha\in\Delta}\mathfrak g_{-\alpha})$,
and consider the variety $\mathcal Z=G\times^B H$,
equipped with the map $\mathcal Z\to\mathfrak g$ given by $(g,x)\mapsto Ad(g)x$.
For $x$ a regular semisimple element of $\mathfrak g$,
the fiber $\mathcal Z_x$ has dimension $|\Delta|$; see \cite[Cor 3]{precup}.
Since regular semisimple elements are dense in $\mathfrak g$,
it follows from \cite[Ch.~1, \S8, Thms. 2, 3]{mumford} that each irreducible component
of the fiber $\mathcal Z_e=\Pet(\Delta)$ has dimension greater than or equal to $|\Delta|$,
and hence $\Pet(\Delta)$ is pure-dimensional.
Following \cite[\S 1.5]{fulton:IT},
the fundamental classes of the irreducible components of $\Pet(\Delta)$ freely generate the top Chow group of $\Pet(\Delta)$.
Since there is a unique top-dimensional cell in the affine paving of \cref{prop:paving}(b),
it follows that $\Pet(\Delta)$ has a unique irreducible component.
\end{rem}

\section{Intersection Multiplicities for classical diagrams}
\label{classicalCalculations}
We present here the details of our calculation in \cref{ClassicalFormulae} of the intersection multiplicities $m(v_I)$ for classical Dynkin diagrams.

\subsection{Type A}
\label{Acalc}
Let $V$ be a vector space with orthonormal basis $\epsilon_1,\cdots,\epsilon_n$.
The vectors $\{\epsilon_i-\epsilon_j\}$ form a root system with Dynkin diagram $A_{n-1}$.
A choice of simple roots is $\alpha_i=\epsilon_i-\epsilon_{i+1}$ for $1\leq i<n$,
and the Weyl group is naturally identified with the symmetric group on $\{\epsilon_1,\cdots,\epsilon_n\}$.
We calculate $v_Iw_I=[ 1,n,\cdots,2 ]$, and
$
\set{\alpha\in\Phi_I ^+}{s_\alpha\not\leq v_Iw_I}=\set{\epsilon_1-\epsilon_i}{2\leq i\leq n}.
$
Now, $ht(\epsilon_1-\epsilon_i)=i-1$.
Consequently \Cref{rootHts,table:exponents} yield $m(v_I)=1$.

\subsection{Type B}
\label{Bcalc}
Let $V$ be a vector space with orthonormal basis $\epsilon_1,\cdots,\epsilon_n$.
The vectors $\{\pm\epsilon_i\pm\epsilon_j\}\cup\{\pm\epsilon_i\}$ form a root system with Dynkin diagram $B_n$.
A choice of simple roots is $\alpha_i=\epsilon_i-\epsilon_{i+1}$ for $i<n$, and $\alpha_n=\epsilon_n$.

Let $S_{2n}$ be the symmetric group on the letters $\{1,\cdots,n,\bar n,\cdots,\bar 1\}$,
and let $r_{i\,j}\in S_{2n}$ be the transposition switching the letters $i$ and $j$.
The Weyl group $W$ can be viewed as the subgroup of $S_{2n}$ generated by the reflections,
\begin{align*}
s_{\epsilon_i-\epsilon_j} & = r_{i\,j} r_{\bar i\,\bar j},& 
s_{\epsilon_i+\epsilon_j} & = r_{i\,\bar j} r_{\bar i\,j},&  
s_{\epsilon_i}            & = r_{i\,\bar i},&1\leq i<j\leq n. 
\end{align*}
Given $v,w\in W$, if $v\leq w$ in the Bruhat order on $W$,
then $v\leq w$ in the Bruhat order on $S_{2n}$; see \cite[\S6.1.1]{lakshmibai.raghavan}.
We calculate \begin{align*}v_Iw_I=[ \bar 2,\cdots, \bar n, 1,\bar 1,n,\cdots,2 ]\end{align*} and
$
\set{\alpha\in\Phi_I ^+}{s_\alpha\not\leq v_Iw_I}=\set{\epsilon_1+\epsilon_i}{2\leq i\leq n}\cup\{\epsilon_1\}.
$
Now, $ht(\epsilon_1)=n$, and $ht(\epsilon_1+\epsilon_i)=2n+1-i$.
Following \Cref{rootHts,table:exponents}, we have
\begin{align*}
m(v_I)=\frac{n(n+1)\cdots(2n-1)}{(1)(3)\cdots(2n-1)}=2^{n-1}.
\end{align*}

\subsection{Type C}
\label{Ccalc}
Let $V$ be a vector space with orthonormal basis $\epsilon_1,\cdots,\epsilon_n$.
The set of vectors $\{\pm\epsilon_i\pm\epsilon_j\}\cup\{\pm2\epsilon_i\}$  forms a root system with Dynkin diagram $C_n$.
A choice of simple roots is $\alpha_i=\epsilon_i-\epsilon_{i+1}$ for $i<n$, and $\alpha_n=2\epsilon_n$.
The Weyl group of $C_n$ is isomorphic the Weyl group of $B_n$.
We calculate 
$
\set{\alpha\in\Phi_I ^+}{s_\alpha\not\leq v_Iw_I}=\set{\epsilon_1+\epsilon_i}{2\leq i\leq n}\cup\{2\epsilon_1\}.
$
Now, $ht(2\epsilon_1)=2n-1$, and $ht(\epsilon_1+\epsilon_i)=2n-i$, for $2\leq i\leq n$.
Following \Cref{rootHts,table:exponents}, we have
\begin{align*}
m(v_I)=\frac{n(n+1)\cdots(2n-1)}{(1)(3)\cdots(2n-1)}=2^{n-1}.
\end{align*}

\subsection{Type D}
\label{Dcalc}
Let $V$ be a vector space with orthonormal basis $\epsilon_1,\cdots,\epsilon_n$.
The set of vectors $\{\pm\epsilon_i\pm\epsilon_j\}$ forms a root system with Dynkin diagram $D_n$.
A choice of simple roots is $\alpha_i=\epsilon_i-\epsilon_{i+1}$ for $i<n$, and $\alpha_n=\epsilon_{n-1}+\epsilon_{n+1}$. 

Let $S_{2n}$ be the symmetric group on the letters $\{1,\cdots,n,\bar n,\cdots,\bar 1\}$,
and let $r_{i\,j}\in S_{2n}$ be the transposition switching the letters $i$ and $j$.
The Weyl group $W$ can be viewed as the subgroup of $S_{2n}$ generated by the reflections,
\begin{align*}
&&s_{\epsilon_i-\epsilon_j} & = r_{i\,j}, &
s_{\epsilon_i+\epsilon_j} & = r_{i\,\bar j} r_{\bar i,\,j},&
1\leq i<j\leq n.
\end{align*}
Given $v,w\in W$, if $v\leq w$ in the Bruhat order on $W$,
then $v\leq w$ in the Bruhat order on $S_{2n}$; see \cite[\S7.1.1]{lakshmibai.raghavan}.
A simple computation yields
\begin{align*}
v_Iw_I=\begin{cases}
\left[ \bar 2,\cdots, \bar{n-1}, 1, n \right]      & \text{if }n\text{ is even},\\
\left[ \bar 2,\cdots, \bar{n-1}, 1 ,\bar n \right] & \text{if }n\text{ is odd}.
\end{cases}
\end{align*}
Observe that
$
\set{\alpha\in\Phi_I ^+}{s_\alpha\not\leq v_Iw_I}=\set{\epsilon_1+\epsilon_i}{2\leq i\leq n}\cup\{\epsilon_1-\epsilon_n\}.
$
Now $ht(\epsilon_1-\epsilon_n)=n-1$, and $ht(\epsilon_1+\epsilon_i)=2n-1-i$, for $2\leq i\leq n$.
Consequently, we deduce from \cref{rootHts,table:exponents} that
\begin{align*}
m(v_I)=\frac{(n-1)\, (n-1) n \cdots (2n-3)}{(1)(3)\cdots(2n-3)\,(n-1)}
    =\frac{(n-1) \cdots (2n-3)}{(1)(3)\cdots(2n-3)}=2^{n-2}.
\end{align*}

\bibliographystyle{amsalpha}
\bibliography{biblio}

\end{document}